\documentclass[11pt]{article}
\usepackage{amsfonts}
\usepackage{graphics}
\usepackage{indentfirst}
\usepackage{color}
\usepackage{cite}
\usepackage{float}
\usepackage{caption}
\usepackage{latexsym}
\usepackage[paper=a4paper, left=1.6cm, right=1.6cm, top=1.8cm, bottom=1.6cm, headheight=5.5pt, footskip=0.8cm, footnotesep=0.8cm, centering, includefoot]{geometry}
\usepackage{amsmath}
\allowdisplaybreaks
\usepackage{amssymb}
\usepackage[colorlinks, linkcolor=red]{hyperref}
\hypersetup{urlcolor=red, citecolor=red}
\usepackage[dvips]{epsfig}
\usepackage{amscd}

\usepackage{amsthm}
\usepackage{mathrsfs}
\usepackage{verbatim}
\newtheorem{theorem}{Theorem}[section]
\newtheorem{remark}{Remark}[section]
\newtheorem{definition}{Definition}[section]
\newtheorem{lemma}{Lemma}[section]
\newtheorem{corollary}{Corollary}[section]
\newtheorem{proposition}{Proposition}[section]

\allowdisplaybreaks

\renewcommand{\div}{{\rm div}}

\DeclareMathOperator{\loc}{loc}

\makeatletter
\@addtoreset{equation}{section}
\makeatother
\makeatletter
\@addtoreset{equation}{section}
\makeatother

\title{ Pressure effects on critical scaling and global low-regularity  solutions for compressible Navier--Stokes system
\thanks{This research was partially supported by National Natural Science Foundation of China (No. 12371227) and Fundamental Research Funds for the Central Universities (No. SWU--KU24001).}
}

\author{Lin Xu,\ Xin Zhong{\thanks{Corresponding author. E-mail addresses: mathxu@email.swu.edu.cn (L. Xu), xzhong1014@amss.ac.cn (X. Zhong).}}
	\date{}\\
	\footnotesize School of Mathematics and Statistics, Southwest University, Chongqing 400715, P. R. China}
\begin{document}
\maketitle

\begin{abstract}
	This paper investigates the three-dimensional compressible Navier--Stokes system with a polytropic pressure law and its pressureless counterpart arising from the high Mach number limit. We focus on the different critical scaling structures of these two models. In the presence of the pressure term, the pressure gradient is balanced with  the inertial and viscous effects, and thereby selects a fixed critical scaling for the pressure system. In contrast, once the pressure term is ignored, the pressureless system admits a more flexible one-parameter family of invariant scalings.
	
	For both systems, we establish the global well-posedness of strong solutions under low-regularity assumptions on the initial data, allowing vacuum and large oscillations. This improves the global result of Wen (Adv. Math. 482 (2025), Paper No. 110628), where higher regularity assumptions on the initial data are required. A central feature of our result is that the smallness conditions are {\it exactly invariant} under the intrinsic critical scalings of the corresponding systems. These scaling structures differ from the usual parabolic scaling used in the critical-space framework of Danchin (Invent. Math. 141 (2000), pp. 579--614), where the system is reformulated around a reference state. We also derive uniform a \textit{priori} estimates and obtain exponential decay estimates for the global strong solutions. The results show that the pressure term not only changes the analytic estimates, but also plays a decisive role in selecting the critical scaling structure and in determining the dynamical behavior of compressible flows.
\end{abstract}

\textit{Key words and phrases}. Compressible Navier--Stokes equations; scaling invariant; pressure; vacuum.

2020 \textit{Mathematics Subject Classification}. 35Q35; 35R05; 76N06; 76N10.

\setcounter{tocdepth}{2}
\tableofcontents

\section{Introduction}

\subsection{Compressible Navier--Stokes system with pressure}
The compressible Navier--Stokes system is one of the fundamental models in
fluid dynamics and describes the motion of viscous compressible fluids. The
mathematical analysis of this system has attracted extensive attention over the past several decades. Nevertheless, some  problems remain largely
open, such as the global existence of smooth solutions for general large
initial data and the uniqueness of finite-energy weak solutions.

Let $\rho=\rho(x,t)\geq 0$ and $u=u(x,t)\in\mathbb{R}^{3}$ denote the density and velocity of the fluid, respectively. The first model is the three-dimensional compressible Navier--Stokes system with pressure,
\begin{align}\tag{NS}\label{NS}
	\left\{
	\begin{aligned}
		&\rho_{t}+\operatorname{div}(\rho u)=0, \\[2mm]
		&(\rho u)_{t}+\operatorname{div}(\rho u\otimes u)
		-\mu\Delta u-(\lambda+\mu)\nabla\operatorname{div}u+\nabla P(\rho)=0,
	\end{aligned}
	\right.
\end{align}
where the pressure $P(\rho)$ is given by the polytropic law
\begin{align*}
	P=A\rho^{\gamma}, \qquad A>0,\quad \gamma>1.
\end{align*}
Here $\mu$ and $\lambda$ are viscosity coefficients satisfying the  physical restrictions
\begin{align*}
	\mu>0,\qquad 2\mu+3\lambda\geq 0.
\end{align*}

We now recall the scaling viewpoint associated with the critical space theory. This viewpoint is important because it reveals the natural critical regularity dictated by the scaling of the equation. Roughly speaking, a function space is called critical if
its norm is invariant under the natural scaling of the equation. To illustrate
this idea, let us first consider the incompressible Navier--Stokes equations
\begin{align}\tag{INS}\label{INS}
	\partial_{t} v-\Delta v+v\cdot\nabla v+\nabla \pi=0,
	\qquad
	\operatorname{div}v=0.
\end{align}
If $(v,\pi)$ is a solution to \eqref{INS}, then, for any $\xi>0$,
\begin{align}\label{INS-scaling}
	v^{\xi}(t,x)\triangleq \xi v(\xi^{2}t,\xi x),
	\qquad
	\pi^{\xi}(t,x)\triangleq \xi^{2}\pi(\xi^{2}t,\xi x)
\end{align}
is also a solution to \eqref{INS}. Hence a space $X$ for the initial velocity is
critical if
\begin{align*}
	\|v_{0}^{\xi}\|_{X}=\|v_{0}\|_{X},
	\qquad
	v_{0}^{\xi}(x)=\xi v_{0}(\xi x).
\end{align*}
Typical examples include $\dot{H}^{1/2}(\mathbb{R}^{3})$,
$L^{3}(\mathbb{R}^{3})$, $\dot B^{-1+3/p}_{p,q}(\mathbb{R}^{3})$, and
$\mathrm{BMO}^{-1}(\mathbb{R}^{3})$. Fujita and Kato \cite{FK64} proved the
well-posedness of \eqref{INS} in $\dot{H}^{1/2}(\mathbb{R}^{3})$. We refer to
\cite{K84,C1997,P1996,KT01} and the references therein for well-posedness
results in other critical spaces.

For compressible Navier--Stokes equations, the critical-space theory is more
delicate because the density and the pressure have to be taken into account. In the momentum equation, the inertial term, the viscous terms, and the pressure gradient must have the same scaling order. One may introduce the parabolic scaling
\begin{align}\label{com}
	\rho^{\xi}(t,x)\triangleq \rho(\xi^{2}t,\xi x),
	\qquad
	u^{\xi}(t,x)\triangleq \xi u(\xi^{2}t,\xi x),
	\qquad
	P^{\xi}(t,x)\triangleq \xi^{2}P(\xi^{2}t,\xi x).
\end{align}
With this transformation, the system keeps the same formal structure if the pressure is transformed as an independent variable. In the framework of critical spaces, Danchin \cite{D2000} made a breakthrough by establishing the global well-posedness of strong solutions to \eqref{NS} in $L^{2}$-type critical Besov spaces for initial data close to a constant equilibrium state. This result showed that the critical functional framework is well adapted to the mixed hyperbolic-parabolic structure of the compressible system. Danchin's critical Besov framework was subsequently extended in several directions. Charve--Danchin \cite{CD2010}, Chen--Miao--Zhang \cite{CMZ2010}, Haspot \cite{H2011b,H2011a}, and Guo--Song--Yang \cite{GSY2025} developed the theory in critical $L^{p}$-type Besov spaces and related larger critical frameworks. In addition, optimal time-decay estimates and large-time behavior in critical spaces have been studied in \cite{X2019,DX2017,XX2021}. These developments show that the scaling-invariant viewpoint is not merely a technical issue, but is closely related to the intrinsic structure of the equations.

The above critical-space results are obtained near a non-vacuum state. More precisely, the density is assumed to be a small perturbation of a positive constant equilibrium. Hence the density is  strictly away from vacuum. Once vacuum is allowed, the momentum equation happens degenerate in the region where the density vanishes. This degeneracy cannot be treated directly by the standard theory of degenerate PDEs, and estimates adapted to the special structure of  system are required. In this direction, a major breakthrough was achieved by Lions \cite{P1998}, and later refined in \cite{FNP01,F2004}, where the global existence of finite-energy weak solutions with vacuum was established under suitable assumptions on the adiabatic exponent $\gamma$. For strong or classical solutions with vacuum, Cho--Choe--Kim \cite{CCK04} proved the local existence and uniqueness of strong solutions under the initial regularity assumption
\begin{align}\label{re1}
	\left( \rho_{0}-\rho_{\infty}\right)  \in H^{1} \cap W^{1,q},
	\qquad
	u_{0} \in D_{0}^{1} \cap D^{2},
\end{align}
where $\rho_{\infty}\in[0,\infty)$ is a constant, together with the compatibility
condition
\begin{align}\label{re2}
	-\mu \Delta u_{0}-(\mu+\lambda)\nabla\operatorname{div} u_{0}
	+\nabla P(\rho_{0})=\sqrt{\rho_{0}}g
\end{align}
for some $g\in L^{2}$. Unlike the aforementioned non-vacuum perturbation problems in critical spaces, well-posedness results in the presence of vacuum are usually developed in inhomogeneous Sobolev spaces. One reason is that the velocity may fail to be $L^{2}$-integrable in vacuum regions; see \cite{LWX2019}. The compatibility condition \eqref{re2} is widely used in the vacuum problem.
This naturally raises the question of whether the regularity assumptions in
\eqref{re1} can be relaxed and whether the compatibility condition \eqref{re2}
is necessary for local well-posedness. Gong--Li--Liu--Zhang \cite{GLLZ20} and Huang \cite{H21} independently removed the initial compatibility condition \eqref{re2} and established local well-posedness for \eqref{NS}. Another important research concerns global strong or classical solutions with vacuum and large density oscillations.  Huang--Li--Xin \cite{HLX12} proved the global well-posedness of classical solutions to \eqref{NS} with small initial energy, allowing vacuum and large oscillations in the initial density. Hong--Hou--Peng--Zhu \cite{HHPZ2024} further generalized this direction by proving the global existence and
uniqueness of classical solutions with large initial energy in the nearly isothermal case, namely when the adiabatic exponent $\gamma$ is sufficiently close to $1$. We also refer to \cite{FZZ2018,WWZ14,HHW2019,ZLZ2020} for other types of large-solution results.

{\it We emphasize that the parabolic pressure scaling in \eqref{com} is not compatible with the polytropic pressure law $P=A\rho^{\gamma}$ considered in this paper.} For \eqref{NS} with the polytropic pressure law, the pressure gradient must have the same scaling order as the inertial and viscous terms in the momentum equation. This requirement fixes a specific critical scaling for the pressured system. Lei--Xin \cite{LX19} observed the following scaling invariant property for \eqref{NS}:
\begin{align}\label{pressure-scaling}
	\rho^{\ell}(x,t)
	\triangleq
	\ell^{\frac{2}{\gamma+1}}
	\rho\left(\ell x,\ell^{\frac{2\gamma}{\gamma+1}}t\right),
	\qquad
	u^{\ell}(x,t)
	\triangleq
	\ell^{\frac{\gamma-1}{\gamma+1}}
	u\left(\ell x,\ell^{\frac{2\gamma}{\gamma+1}}t\right),
\end{align}
where $\ell>0$ is the scaling parameter. The dependence of the scaling exponents on $\gamma$ reflects the nonlinear structure of the pressure law $P=A\rho^{\gamma}$. The verification of
this scaling is given in the Appendix. More recently, Wen\cite{W25} constructed a scaling-invariant initial quantity $N_{0}$ which is
invariant under the transformation \eqref{pressure-scaling} ( if $N_{0} \leq \varepsilon_{*}$ holds for some initial data $(\rho_{0},u_{0})$, then, the rescaled initial data $(\rho^{\ell}_{0},u^{\ell}_{0})$ defined by \eqref{pressure-scaling} at $t=0$ also satisfies $N_{0} \leq \varepsilon_{*}$), and proved the global existence of strong solutions to \eqref{NS} provided that this scaling-invariant initial quantity is sufficiently small.

\subsection{Pressureless compressible Navier--Stokes system}
The second model studied in this paper is the corresponding pressureless Navier--Stokes system. It can be formally derived from  system \eqref{NS} in the high Mach number regime, where the effect of the pressure term becomes negligible. To avoid ambiguity in notation, we denote the original density and velocity by $\varrho=\varrho(\tau,x)$ and $U=U(\tau,x)$, respectively. The compressible  Navier--Stokes equations in $\mathbb{R}^{3}$ as
\begin{align*}
	\left\{
	\begin{aligned}
		&\partial_{\tau}\varrho+\operatorname{div}(\varrho U)=0,\\[2mm]
		&\partial_{\tau}(\varrho U)
		+\operatorname{div}(\varrho U\otimes U)
		+\nabla P(\varrho)-\mu_{*}\Delta U-(\mu_{*}+\lambda_{*})\nabla\operatorname{div}U
		=0,
	\end{aligned}
	\right.
\end{align*}
where $P(\varrho)=A\varrho^{\gamma}$ and the viscosity coefficients satisfy $\mu_{*}>0$, $2\mu_{*}+3\lambda_{*}\geq 0$. Let $\varepsilon>0$ denote the Mach number, and the high Mach number regime corresponds to the limit $\varepsilon\to+\infty$. We introduce the rescaled variables
\begin{align*}
	\rho_{\varepsilon}(t,x)
	\triangleq\varrho\left(\frac{t}{\varepsilon},x\right),
	\quad
	u_{\varepsilon}(t,x)
	\triangleq\frac{1}{\varepsilon}
	U\left(\frac{t}{\varepsilon},x\right),\quad
	\mu_{\varepsilon}\triangleq\frac{\mu_{*}}{\varepsilon},
	\quad
	\lambda_{\varepsilon}\triangleq\frac{\lambda_{*}}{\varepsilon},
\end{align*}
and define
\begin{align*}
	\rho_{\varepsilon}\to \rho,\qquad
	u_{\varepsilon}\to u,
	\qquad
	\mu_{\varepsilon}\to \mu,\qquad
	\lambda_{\varepsilon}\to \lambda.
\end{align*}
By a direct calculation, we obtain
\begin{align*}
	\begin{cases}
		\partial_t\rho_{\varepsilon}
		+\operatorname{div}(\rho_{\varepsilon}u_{\varepsilon})=0,\\[2mm]
		\partial_t(\rho_{\varepsilon}u_{\varepsilon})
		+\operatorname{div}(\rho_{\varepsilon}u_{\varepsilon}\otimes u_{\varepsilon})
		+\frac{1}{\varepsilon^{2}}\nabla P(\rho_{\varepsilon})-\mu_{\varepsilon}\Delta u_{\varepsilon}
		-(\mu_{\varepsilon}+\lambda_{\varepsilon})
		\nabla\operatorname{div}u_{\varepsilon}
		=0.
	\end{cases}
\end{align*}
Now we let $\varepsilon\to+\infty$, the pressure term  vanishes in the high Mach number limit.
Passing formally to the limit in the above system, we obtain the pressureless Navier--Stokes system
\begin{align}\tag{PNS}\label{PNS}
	\left\{
	\begin{aligned}
		&\partial_{t}\rho+\operatorname{div}(\rho u)=0, \\[2mm]
		&\partial_{t}(\rho u)+\operatorname{div}(\rho u\otimes u)-\mu\Delta u-(\mu+\lambda)\nabla\operatorname{div}u
		=0.
	\end{aligned}
	\right.
\end{align}
Apart from its connection with the high Mach number limit, the system \eqref{PNS} has its own physical relevance. It is used to describe sticky particle dynamics, where particles stick together after collision and move with a common terminal velocity \cite{B94}. It also arises in models of dust or free-particle motion in astrophysics \cite{DD12}, multi-fluid systems \cite{B02}, and collective behavior such as traffic flow \cite{BDDR2008}, where an internal pressure force is not a natural modeling assumption.

The mathematical analysis of \eqref{PNS} has attracted increasing attention in
recent years, especially regarding critical regularity and scaling invariant
conditions. In the small-perturbation framework around a positive constant density, Danchin--Mucha--Tolksdorf \cite{DMT2021} proved the global well-posedness of strong solutions to \eqref{PNS}, assuming that the initial density is a small $L^{\infty}(\mathbb{R}^{3})$ perturbation of a positive constant and that the initial velocity belongs to suitable homogeneous Besov spaces. A notable feature of their result is that the smallness condition involves the quantity
$$
\|u_{0}\|_{\dot{B}_{5/2,1}^{6/5}(\mathbb{R}^{3})}^{1/3}
\|u_{0}\|_{\dot{B}_{10/7,1}^{3/5}(\mathbb{R}^{3})}^{2/3},
$$
which is invariant under the standard velocity scaling $u_0^\lambda(x)=\lambda u_0(\lambda x)$. In a related Sobolev setting, Guo--Tang--Zhao \cite{GTZ2025} established the global well-posedness and stability of classical solutions under smallness assumptions on the initial density perturbation and on the $L^{1}$ norm of the initial velocity. This direction was further developed by Li--Ni--Zhang \cite{LNZ2025}, who established the existence of global strong solutions under weaker critical regularity assumptions in homogeneous Besov spaces, and also proved uniform stability and optimal decay rates for a broader class of $u$. Beyond the small-perturbation regime, Xu \cite{X2022} proved the global well-posedness of \eqref{PNS} for initial data with large density in a multiplier space and sufficiently small velocity in critical Besov spaces.  In contrast to the perturbative results in \cite{DMT2021,GTZ2025}, Wang--Wu--Xu \cite{WWX2025} proved the existence of global Fujita--Kato solutions without imposing a small perturbation assumption on the initial density. Their result allows large variations of the density and initial velocities with critical regularity.

\textit{Although \eqref{NS} and \eqref{PNS} have similar viscous structures, the absence of the pressure gradient in \eqref{PNS} leads to a different invariant structure.} This difference is one of the main motivations of the present work. Once the pressure term is ignored, the pressureless system \eqref{PNS} admits the one-parameter family of invariant scalings
\begin{align}\label{pressureless-scaling}
	\rho^{\ell}(x,t)=\ell^{\tau}\rho(\ell x,\ell^{2-\tau}t),
	\qquad
	u^{\ell}(x,t)=\ell^{1-\tau}u(\ell x,\ell^{2-\tau}t), \qquad \text{for  all }~~ \tau\in\mathbb R.
\end{align}
Thus, the pressure term selects a fixed critical scaling for pressure  model \eqref{NS}, whereas the pressureless model \eqref{PNS} retains a larger scaling freedom. This distinction naturally leads to different scaling-invariant  conditions for the two systems and plays an important role in the global well-posedness analysis carried out in this paper.

\subsection{Motivation}
The motivation of this paper comes from the following observations.

\begin{itemize}
	\item The first observation concerns the relation between critical-space theory	 and energy methods with vacuum. Critical-space results reveal the natural scaling of the equations and provide a sharp functional framework for well-posedness. However, these results are usually formulated as small
	perturbations of a constant equilibrium state and rely on Besov-type regularity. On the other hand, energy methods are more suitable for vacuum and large density oscillations problem, but the scaling structure of the corresponding assumption is often not explicit. In this paper, we try to connect these two viewpoints.
	
	\item The second observation is that the pressure term changes the critical scaling structure of the system. For pressure  system \eqref{NS}, the pressure gradient has to have the same scaling order as the inertial and viscous terms in the momentum equation. This balance fixes the critical scaling \eqref{pressure-scaling}. In contrast, once the pressure term is ignored, the pressureless system \eqref{PNS} admits the one-parameter family of invariant scalings \eqref{pressureless-scaling}. Therefore, the difference between the two models is not merely formal. It directly affects the construction of scaling-invariant quantities used in 	the global-in-time estimates.
\end{itemize}	
 Based on these observations, the main purpose of this paper is to establish global well-posedness of strong solutions for both the pressured and pressureless systems under smallness assumptions which are exactly invariant under their corresponding critical scalings.

\subsection{Main results}
Before stating our results, we introduce some notation that will be used
frequently throughout the paper. For an integer $k\geq0$ and
$1\leq q\leq\infty$, we define
\begin{gather*}
	D^{k,q}(\mathbb{R}^{3}) = \{ u \in L^1_{\loc}(\mathbb{R}^{3}) \mid \|\nabla^k u\|_{L^q} < \infty \}, \ \|u\|_{D^{k,q}} = \|\nabla^k u\|_{L^{q}(\mathbb{R}^{3})},\ \|  u \|_{L^{q}}=\|u \|_{L^{q}(\mathbb{R}^{3})}, \\
	W^{k,q}(\mathbb{R}^{3}) = L^q(\mathbb{R}^{3}) \cap D^{k,q}(\mathbb{R}^{3}), \ D^k(\mathbb{R}^{3}) = D^{k,2}(\mathbb{R}^{3}), \ H^k(\mathbb{R}^{3}) = W^{k,2}(\mathbb{R}^{3}).
\end{gather*}
For simplicity, we take $A=1$ in the pressure law
$P=A\rho^{\gamma}$. We also use the shorthand
notation
\begin{align*}
	\int \cdot \, d x \triangleq \int_{\mathbb{R}^{3}} \cdot \, d x,\quad  \dot{u}\triangleq u_t+u\cdot\nabla u.
\end{align*}

The strong solutions to be established in this paper are defined as follows.

\begin{definition}\label{def}
	Let $T>0$ and $q\in(3,6)$. Assume that the initial data
	$(\rho_0\geq 0,u_0)$ satisfies
	\begin{align}\label{sou}
		\left( \rho_{0}-\rho_{\infty}\right) \in   H^{1}\cap W^{1,q},\qquad u_{0}\in D^{1}.
	\end{align}
	A pair $(\rho,u)$ is called a strong solution to system  \eqref{NS}  or \eqref{PNS}  in $\mathbb{R}^{3}\times[0,T]$ if it satisfies the following regularity properties:
	\begin{align}\label{defs}
		\begin{cases}
			(\rho-\rho_{\infty})\in C\big([0,T];L^{2}\big)
			\cap L^{\infty}\big(0,T;H^{1}\cap W^{1,q}\big),
			\quad
			\rho_{t}\in L^{\infty}\big(0,T;L^{2}\big), \\
			\rho u\in C\big([0,T];L^{2}\big),\quad
			u\in L^{\infty}\big(0,T;D^{1}\big)
			\cap L^{2}\big(0,T;D^{2}\big),
			\quad
			\sqrt{\rho} u_{t}\in L^{2}\big(0,T;L^{2}\big), \\
			\sqrt{t} u\in L^{\infty}\big(0,T;D^{2}\big)
			\cap L^{2}\big(0,T;D^{2,q}\big),
			\quad
			\sqrt{t} u_{t}\in L^{2}\big(0,T;D^{1}\big).
		\end{cases}
	\end{align}
	Moreover, $(\rho,u)$ satisfies \eqref{NS}  or \eqref{PNS}   a.e. in $\mathbb{R}^{3}\times(0,T]$. In particular, the strong solution  $(\rho, u )$  of \eqref{NS}  or \eqref{PNS}  is called the global strong solution, if the strong solution satisfies \eqref{defs} for any  $T>0$, and satisfies \eqref{NS}  or \eqref{PNS}  a.e. in  $\mathbb{R}^{3} \times(0, T)$.
\end{definition}

\subsubsection{The pressureless system}
First, we consider the Cauchy problem for the pressureless system \eqref{PNS},
supplemented with the initial data
\begin{equation}\label{a2}
	(\rho,\rho u)(0,x)=\bigl(\rho_{0},\rho_{0}u_{0}\bigr)(x),
	\quad x\in\mathbb{R}^{3},
\end{equation}
and the far-field condition
\begin{equation}\label{a3}
	\lim_{|x|\rightarrow\infty}(\rho,u)(t,x)=(\rho_{\infty},0),
\end{equation}
where $\rho_{\infty}\geq 0$ is a given constant. Now we are in a position to state our first result.
\begin{theorem}\label{th1}
Let $q\in(3,6)$ and assume that the initial data $(\rho_{0}\geq 0,u_{0})$ satisfies
	\begin{align}\label{xz}
		(\rho_{0}-\rho_{\infty})\in \left\{\begin{array}{ll}
		L^{\frac{3}{2}} \cap H^{1} \cap W^{1,q},&\rho_{\infty}=0,\\
	  H^{1} \cap W^{1,q},&\rho_{\infty}>0,
	\end{array}\right.	
	\quad   u_{0} \in D^{1}.
	\end{align}
 There exists a positive constant $\varepsilon_{0}$ depending only on the parameters $\mu$ and $\lambda$ such that if
	\begin{align}\label{csca1}
	 \|\rho_{0} \|_{L^{\infty}}^{3} \| \sqrt{\rho_{0}}u_{0}\|_{L^{2}}^{2}\|\nabla u_{0}\|_{L^{2}}^{2} \leq \varepsilon_{0},
	\end{align}
	then the Cauchy problem \eqref{PNS}, \eqref{a2}, and \eqref{a3} admits  a unique global strong solution $(\rho,u )$.
	
Moreover, in the case $\rho_{\infty}=0$, there exist positive constants $C$ and
$\alpha$ independent of $t$ such that, for all $t\geq 1$,
	\begin{align}\label{cexp}
		&\|\rho_{t} \|_{L^{2}}^{2}+\|\sqrt{\rho} u\|_{L^{2}}^{2} +\|\nabla u\|_{H^{1}}^{2} + \|\sqrt{\rho} \dot{u} \|_{L^{2}}^{2} \leq C e^{-\alpha t}.
	\end{align}
\end{theorem}

\begin{remark}
By Definition \ref{def}, \eqref{xz}, and Sobolev's inequality, one has
$\rho_{0}u_{0}\in L^{2}$. Thus, the initial momentum is well-defined even when
vacuum is allowed. It should also be emphasized that the constant
$\varepsilon_{0}$ appearing in the smallness condition \eqref{csca1} is
independent of the initial data, and that the condition itself is compatible
with the scaling invariant structure of the pressureless system.
\end{remark}

\begin{remark}
We next explain the scaling invariance of the smallness condition
\eqref{csca1}. For the pressureless system with the far-field state
$\rho_{\infty}>0$, the scaling must preserve the non-zero constant state at
infinity. This corresponds to choosing the scaling exponent $\tau=0$ in
\eqref{pressureless-scaling}. In this case, the initial data are transformed as
\begin{align*}
	\rho_0^\ell(x)=\rho_0(\ell x),
	\qquad
	u_0^\ell(x)=\ell u_0(\ell x).
\end{align*}
Since the coefficient multiplying $\rho$ is of zeroth order, we can deal with non-vacuum far-field state. Therefore, in the case $\rho_{\infty}>0$, the Cauchy problem
\eqref{PNS}, \eqref{a2}, and \eqref{a3} is invariant under the scaling
\begin{align*}
	\rho^{\ell}(x,t)= \rho(\ell x,\ell^{2 }t),
	\qquad
	u^{\ell}(x,t)=\ell u(\ell x,\ell^{2 }t).
\end{align*}

In particular,  for far-field vacuum case, the smallness condition \eqref{csca1} is invariant for  all $\tau\in\mathbb R$. More exactly, the initial data are transformed as
\begin{align*}
	\rho_0^\ell(x)=\ell^{\tau}\rho_0(\ell x),
	\qquad
	u_0^\ell(x)=\ell^{\,1-\tau}u_0(\ell x).
\end{align*}
Then
\begin{align*}
	\|\rho_0^\ell\|_{L^\infty}^{3}
	\|\sqrt{\rho_0^\ell}u_0^\ell\|_{L^2}^{2}
	\|\nabla u_0^\ell\|_{L^2}^{2}
	&=\ell^{3\tau}\ell^{-1-\tau}\ell^{1-2\tau}
	\|\rho_0\|_{L^\infty}^{3}
	\|\sqrt{\rho_0}u_0\|_{L^2}^{2}
	\|\nabla u_0\|_{L^2}^{2}\\
	&=
	\|\rho_0\|_{L^\infty}^{3}
	\|\sqrt{\rho_0}u_0\|_{L^2}^{2}
	\|\nabla u_0\|_{L^2}^{2}.
\end{align*}
Hence the smallness condition \eqref{csca1} is invariant under the whole
pressureless scaling family when $\rho_{\infty}=0$.
\end{remark}

\begin{remark}
It is worth emphasizing that the smallness assumption is imposed on the
scaling invariant product in \eqref{csca1}. Thus, one of the factors may be large provided that the other factors compensate for it. For instance, the initial kinetic energy
$\|\sqrt{\rho_{0}}u_{0}\|_{L^{2}}^{2}$ may be large if
$\|\rho_{0}\|_{L^{\infty}}$ or $\|\nabla u_{0}\|_{L^{2}}$ is sufficiently small.
In this sense, our theorem includes a class of large solutions. Of course, this is still far from
establishing global solutions for general large initial data.
\end{remark}

\begin{remark}
Compared with the global well-posedness results in
\cite{GTZ2025,DMT2021,WWX2025,LNZ2025}, where perturbations around non-vacuum
states were considered in critical Besov spaces, our result allows the initial
density to contain vacuum and large oscillations. Compared with the classical
results on large density oscillations and vacuum for \eqref{NS}, such as \cite{HLX12}, our pressureless result gives not only global existence and uniqueness, but also the exponential decay
estimate \eqref{cexp}. This indicates that the
presence or absence of the pressure term has a significant influence on the
dynamical behavior and stability properties of compressible flows.
\end{remark}

\begin{remark}
	The main limitation of our approach is that it does not directly apply to	general domains with boundaries, since such domains are usually not preserved under the scaling. Nevertheless, the half-space case is still compatible with our argument. Similar boundary issues have been studied for the inhomogeneous incompressible
	Navier--Stokes equations; see \cite{DM2009,DZ2014}. Indeed, in the half-space
	\begin{align*}
	\mathbb{R}^{3}_{+}=\{x=(x_1,x_2,x_3)\in\mathbb{R}^{3}:x_3>0\},
\end{align*}
the pressureless scaling preserves both the domain and its boundary, namely
	\begin{align*}
	\ell\mathbb{R}^{3}_{+}=\mathbb{R}^{3}_{+},
	\qquad
	\ell\partial\mathbb{R}^{3}_{+}=\partial\mathbb{R}^{3}_{+}.
\end{align*}
	Consequently, if $u|_{\partial\mathbb{R}^{3}_{+}}=0$, then also $u^\ell|_{\partial\mathbb{R}^{3}_{+}}=0$.
	Thus, in the half-space, the above scaling \eqref{pressureless-scaling} remains compatible with the pressureless system under the no-slip boundary condition. This problem will be addressed in a forthcoming work.
\end{remark}

\subsubsection{The pressure system}
The system  \eqref{NS}   is supplemented with the initial data
\begin{align}\label{ca2}
	(\rho,\rho u)(0,x)=\bigl(\rho_{0},\rho_{0}u_{0}\bigr)(x),
	\qquad x\in\mathbb{R}^{3},
\end{align}
and the far-field condition
\begin{align}\label{ca3}
	\lim_{|x|\rightarrow\infty}(\rho,u)(t,x)=(0,0).
\end{align}
We are in a position to state our second main result.
\begin{theorem}\label{th2}
 Let  $q \in(3,6)$  and assume that the initial data $(\rho_{0}\ge 0,u_{0})$ satisfies
	\begin{align}\label{ini}
		 \rho_{0} \in L^{\gamma} \cap H^{1} \cap W^{1,q}, \quad u_{0} \in D^{1}.
	\end{align}
	There exists a small positive constant $\varepsilon_{0}^{\prime}$ depending only on the parameters $\mu$ and $\lambda$ such that if
\begin{align}\label{sca1}
	&\bar{\rho}^{3}
	\left(
	\frac{1}{2}\|\sqrt{\rho_{0}}u_{0}\|_{L^{2}}^{2}
	+\frac{1}{\gamma-1}\|P_{0}\|_{L^{1}}
	\right)
	\left(
	\|\nabla u_{0}\|_{L^{2}}^{2}
	+\|P_{0}\|_{L^{2}}^{2}
	\right) \notag\\
	& \times
	\left[
	1+\bar{\rho}^{3+\gamma}
	\left(
	\frac{1}{2}\|\sqrt{\rho_{0}}u_{0}\|_{L^{2}}^{2}
	+\frac{1}{\gamma-1}\|P_{0}\|_{L^{1}}
	\right)^{2}
	\right]
	\leq \varepsilon_{0}^{\prime},
\end{align}
where  $\bar{\rho}\triangleq \|\rho_{0} \|_{L^{\infty}}$, then the Cauchy problem \eqref{NS}, \eqref{ca2}, and \eqref{ca3} admits a unique global strong solution $(\rho,u)$.
\end{theorem}

\begin{remark}
The initial quantity in \eqref{sca1} is scaling invariant for \eqref{pressure-scaling}. For pressure system \eqref{NS}, the scaling structure is more
delicate, since the pressure law have to be taken into account
simultaneously.  Even if the pressure-related terms are ignored from \eqref{sca1}, the remaining part still does not match \eqref{csca1}. This naturally raises the question whether there exists a simpler scaling invariant quantity adapted to the pressure system \eqref{NS}.
\end{remark}

\begin{remark}
	The scaling invariant property \eqref{pressure-scaling} contains the factor
	$\ell^{\frac{2}{\gamma+1}}$ in front of the density. Hence, this scaling is not
	compatible with non-vacuum	far-field case.
\end{remark}

\begin{remark}
	Under the low-regularity assumption \eqref{ini} on the initial data, the local
	well-posedness theory was established in \cite{H21,GLLZ20}. Theorem
	\ref{th2} builds on this theory and extends the corresponding local strong
	solution to a global one. In	comparison with the global well-posedness result obtained by Wen \cite{W25}	under the higher-regularity assumptions \eqref{re1} and \eqref{re2}, our result is proved under
	lower regularity assumption \eqref{ini}.
\end{remark}

\begin{remark}
	For the pressureless system \eqref{PNS}, the scaling admits a one-parameter family
	\begin{align*}
	\rho^\ell(x,t)=\ell^\tau\rho(\ell x,\ell^{2-\tau}t),
	\qquad
	u^\ell(x,t)=\ell^{1-\tau}u(\ell x,\ell^{2-\tau}t).
\end{align*}
	When the barotropic pressure $P(\rho)= \rho^\gamma$ is included, the pressure term fixes the free parameter as
	\begin{align*}
	\tau=\frac{2}{\gamma+1}.
\end{align*}
	Hence, the scaling of the barotropic system \eqref{NS} can be regarded as a special case of the pressureless \eqref{PNS} scaling family.
\end{remark}

\subsection{Strategy of the proofs}

We briefly explain the main ideas of the proofs. The proofs are based on a
combination of scaling invariant lower-order estimates, time weighted
higher-order estimates, and a standard continuation argument. The main
difficulty is to keep the smallness assumptions consistent with the intrinsic
scaling structures of the corresponding systems, while allowing vacuum and
large oscillations of the initial density.

For the lower-order estimates, we use a bootstrap argument (see Proposition \ref{pro}). The key point is
to prove estimates which are uniform in time and depend only on the
scaling invariant quantities of the initial data. In the pressureless system \eqref{PNS}, the
key quantity is
	\begin{align}\label{key}
\bar{\rho}^{3}\cdot\|\sqrt{\rho}u\|_{L^{2}}^{2}
\cdot\|\nabla u\|_{L^{2}}^{2}.
\end{align}
 From the basic energy estimate, we immediately derive the bound for $L_{t}^{\infty}L_{x}^{2}$-norm of $\sqrt{\rho}u$.
To obtain the estimate  $L_{t}^{\infty}L_{x}^{2}$-norm of $\nabla u$ (see Lemma \ref{co1}), we using the smallness of
\begin{align}\label{key0}
\bar{\rho}^{3}\|\nabla u\|_{L^{2}}^{2}\int_{0}^{T}\|\nabla u\|_{L^{2}}^{2}\,dt,
\end{align}
which was constructed in \eqref{cp1}. Then, using the smallness of \eqref{key}  (see \eqref{cs4}), the upper bound of the density is
derived by combining the particle-trajectory argument with a commutator
estimate (see Lemma \ref{rho}). In particular, multiplying \eqref{X0} by $\bar{\rho}^{3}$ and combine with \eqref{xx0}, we construct a complete scaling invariant  initial quantity $\|\rho_{0} \|_{L^{\infty}}^{3} \| \sqrt{\rho_{0}}u_{0}\|_{L^{2}}^{2}\|\nabla u_{0}\|_{L^{2}}^{2}$ and close the estimates \eqref{key}, \eqref{key0}. These estimates in Lemmas \ref{rho} and \ref{sca} complete the bootstrap argument  when the condition \eqref{csca1} holds.

Once the lower-order estimates are obtained, we establish  time-weighted higher-order estimates. These estimates are used to justify the time continuity of the strong solution and to guarantee that the local solution can be extended.  The logarithmic Beale--Kato--Majda type inequality (see Lemma \ref{cinf}) plays an important role in controlling
\begin{align}\label{key2}
\int_{0}^{T}\|\nabla u\|_{L^{\infty}}\,dt,
\end{align}
which is needed for the propagation of the $H^{1}\cap W^{1,q}$ regularity of
the density. In order to estimate $\eqref{key2}$, we need to establish time-weighted estimates for $\| \sqrt{\rho} \dot{u} \|_{L^{2}}^{2}$ and $\| \nabla \dot{u} \|_{L^{2}}^{2}$ (see \eqref{clll6}, \eqref{clll7}). Then, the required time-weighted and higher-order estimates for the velocity field are established in Lemma \ref{clem53}. Finally, we apply a continuation argument to prove the strong solution exists globally in time.

For the  pressure system \eqref{NS}, additional estimates
are needed to handle the pressure term. Applying $(-\Delta)^{-1}\operatorname{div}$
to the momentum equation  gives the identity
\begin{align*}
	P
	=
	(-\Delta)^{-1}\operatorname{div}(\rho u)_{t}
	+(-\Delta)^{-1}\operatorname{div}\operatorname{div}(\rho u\otimes u)
	+(2\mu+\lambda)\operatorname{div}u.
\end{align*}
Together with the a \textit{priori} estimates established in \cite{W25}, this identity
allows us to obtain the $L_{t}^{2}L_{x}^{2}$ estimate of $P$; see Proposition
\ref{wen}. Moreover, by mass conservation equation, the pressure satisfies a transport equation
\begin{align*}
	P_t+u\cdot \nabla P+\gamma P \operatorname{div} u=0.
\end{align*}
Combining this equation with the \textit{effective viscous flux}
\begin{align*}
F\triangleq(2\mu+\lambda)\operatorname{div}u-P,
\end{align*}
we obtain, for $p\geq2$,
\begin{align*}
	\frac{d}{dt}\|P\|_{L^{p}}^{p}
	+\frac{p\gamma-1}{2(2\mu+\lambda)}
	\|P\|_{L^{p+1}}^{p+1}
	\leq C\|F\|_{L^{p+1}}^{p+1}.
\end{align*}
This inequality is crucial for controlling the higher integrability of the
pressure and for closing the time-weighted estimates involving
$\sqrt{\rho}\dot u$ and $\nabla\dot u$ (see Lemmas \ref{lem51}, \ref{lem52}). Once these estimates are obtained, the
same continuation argument as in the pressureless case yields the global
existence of strong solutions for the pressured system.

The rest of the paper is organized as follows. Section \ref{sec2} collects some
known facts and auxiliary inequalities used throughout the paper. Section
\ref{sec3} is devoted to the proof of Theorem \ref{th1}; the scaling-invariant
estimates and the time-weighted estimates are established in Subsections
\ref{sub1} and \ref{sub2}, respectively, and the proof is completed in
Subsection \ref{sub3}. Section \ref{sec4} proves Theorem \ref{th2} by following
a similar strategy.

\section{Preliminaries}\label{sec2}
In this section, we collect some known facts and auxiliary inequalities that will be used frequently in the sequel. We begin with the local well-posedness of strong solutions (see \cite[Theorem 1.1]{GLLZ20}).
\begin{proposition}[Local well-posedness]\label{local} Assume that the initial data $(\rho_{0},u_{0})$ satisfies \eqref{sou}. Then there exists a time $T_{0}>0$ such that the problem \eqref{NS}, \eqref{ca2}, and \eqref{ca3} (or problem
	\eqref{PNS}, \eqref{a2}, and \eqref{a3}) admits a unique strong solution in $\mathbb{R}^{3}\times(0,T_{0}]$. \end{proposition}

Next, we define the commutator \begin{align*} [b,R_iR_j](f) \triangleq bR_iR_j(f)-R_iR_j(bf), \qquad i,j=1,2,3, \end{align*} where $R_i$ denotes the usual Riesz transform on $\mathbb{R}^{3}$, namely $R_i=(-\Delta)^{-1/2}\partial_i$. We need the following commutator estimate (see \cite{CM1975}).

\begin{lemma}[Commutator estimate]\label{comm} Let $b\in D^{1,q}$ and $f\in L^{r}$. Assume that $p,q,r\in(1,\infty)$ satisfy
\begin{align*}
\frac{1}{p}=\frac{1}{q}+\frac{1}{r}.
\end{align*}
Then there exists a constant $C$ depending only on $p$, $q$, and $r$ such that
\begin{align*} \left\|\nabla[b,R_iR_j](f)\right\|_{L^{p}} \leq C\|\nabla b\|_{L^{q}}\|f\|_{L^{r}}.
\end{align*}
\end{lemma}

Next, we recall the following Gagliardo--Nirenberg inequality (see \cite[Theorem 12.87]{book17}).
 \begin{lemma}[Gagliardo--Nirenberg inequality]\label{gn}
 	Let  $1 \leq p$, $q \leq \infty$, $m \in \mathbb{N}$, $k \in \mathbb{N}_{0}$, with  $0 \leq k<m$, and let  $a$, $r$  be such that
 	\begin{align*}
 		0 \leq a \leq 1-\frac{k}{m}
 	\end{align*}
 	and
 	\begin{align*}
 		 (1-a)\bigg(\frac{1}{p}-\frac{m-k}{3}\bigg)+a\bigg(\frac{1}{q}+\frac{k}{3}\bigg)=\frac{1}{r} \in[0,1].
 	\end{align*}
 	Then there exists a constant  $C=C(m, p, q, a, k)>0$  such that
 	\begin{align}\label{gn1}
 		\|\nabla^{k} f \|_{L^{r}} \leq C\|f\|_{L^{q}}^{a} \|\nabla^{m} f \|_{L^{p}}^{1-a}
 	\end{align}
 	for every  $f \in L^{q} (\mathbb{R}^{3} ) \cap D^{m, p} (\mathbb{R}^{3} )$, with the following exceptional cases:
 	\begin{itemize}
 		\item[(i)] If $k=0$, $mp<3$, and $q=\infty$, we assume that $f$ vanishes at infinity.
 		\item[(ii)] If $1<p<\infty$ and $m-k-\frac{3}{p}$ is a non-negative integer, then \eqref{gn1} holds only for $0<a\leq 1-\frac{k}{m}$.
 	\end{itemize}
 \end{lemma}

Furthermore, we have the following logarithmic inequality, which may be regarded as a preliminary version of the Beale--Kato--Majda type inequality.
\begin{lemma}\label{cinf} Let $3<q<6$ and assume that $u\in\left\{v\in L_{\mathrm{loc}}^{1} \mid \nabla v\in L^{2}\cap D^{1,q}\right\}$. Then there exists a positive constant $C$ such that
	\begin{align}
\eqref{NS}:\|\nabla u\|_{L^{\infty}} &\leq C\ln\left(e+\|\nabla^{2}u\|_{L^{q}}\right) \left(  \|\operatorname{curl}u\|_{L^{\infty}} +\|\operatorname{div}u\|_{L^{\infty}}\right)  +C\|\nabla u\|_{L^{2}}+C,\label{cii}\\[2mm]
\eqref{PNS}:\|\nabla u\|_{L^{\infty}} &\leq C\ln\left(e+\|\nabla\rho\|_{L^{q}}\right) \left(  \|\operatorname{curl}u\|_{L^{\infty}} +\|\operatorname{div}u\|_{L^{\infty}}\right) +C\|\nabla u\|_{L^{6}} +C\|\rho\dot{u}\|_{L^{q}}.\label{cinf0}
	\end{align}
\end{lemma}
\begin{proof}
For the \eqref{NS} system, \eqref{cii} follows from \cite[Lemma 2.3]{HLX11}.

Now we give the proof of \eqref{cinf0} for the \eqref{PNS} system.
It follows from \cite[(2.9) in Lemma 2.3]{HLX11}, with minor modifications, that \begin{align}\label{cinf4}
		\|\nabla u\|_{L^{\infty}} \leq C\delta^{\frac{q-3}{q}}\|\nabla^{2}u\|_{L^{q}} +C(1-\ln\delta) \left(  \|\operatorname{curl}u\|_{L^{\infty}} +\|\operatorname{div}u\|_{L^{\infty}}\right) +C\|\nabla u\|_{L^{6}},
		\end{align}
		where $\delta\in(0,1]$ is a constant to be chosen.
	
	The elliptic regularity estimate for \eqref{PNS}$_2$ gives
	\begin{align}\label{do1}
		\|\nabla^{2}u\|_{L^{r}} \leq C\|\rho\dot{u}\|_{L^{r}},\quad \text{for}~~r\in[2,6].
		\end{align}
	Substituting \eqref{do1} into \eqref{cinf4} implies that
	\begin{align*}
			\|\nabla u\|_{L^{\infty}} \leq& C(1-\ln\delta) \left(  \|\operatorname{curl}u\|_{L^{\infty}} +\|\operatorname{div}u\|_{L^{\infty}}\right) +C\|\nabla u\|_{L^{6}} +C\delta^{\frac{q-3}{q}}\|\rho\dot{u}\|_{L^{q}},
	\end{align*}
which yields \eqref{cinf0} by choosing
\begin{align*}
\delta \triangleq \left(e+\|\nabla\rho\|_{L^{q}}\right)^{-\frac{q}{q-3}}.\tag*{\qedhere}
\end{align*}
\end{proof}

Finally, we establish some estimates for the \textit{effective viscous flux} $F$ defined by
\begin{align*}
F=(2\mu+\lambda)\operatorname{div}u-P. \end{align*}
Applying $\operatorname{div}$ and $\operatorname{curl}$ to \eqref{NS}$_2$, we derive
\begin{align}\label{xequ}
	\begin{cases} \Delta F=\operatorname{div}(\rho\dot{u}),\\ \mu\Delta\operatorname{curl}u=\operatorname{curl}(\rho\dot{u}).
	\end{cases}
\end{align}
We now state several estimates that will be used frequently in the subsequent proofs. \begin{lemma}[{\cite[Lemma 2.3]{HLX12}}]\label{F} Let $(\rho,u)$ be a smooth solution to \eqref{NS}. For any $p\in[2,6]$, then there exists a generic positive constant $C$, depending only on $p$, $\mu$, and $\lambda$, such that
\begin{gather}
\|\nabla F\|_{L^{p}} +\|\nabla\operatorname{curl}u\|_{L^{p}} \leq C\|\rho\dot{u}\|_{L^{p}}, \label{uu1}\\ \|\nabla u\|_{L^{p}} \leq C\left(\|F\|_{L^{p}}+\|\operatorname{curl}u\|_{L^{p}}\right) +C\|P\|_{L^{p}},\label{uu3}\\ \|F\|_{L^{p}} +\|\operatorname{curl}u\|_{L^{p}} \leq C\|\rho\dot{u}\|_{L^{2}}^{(3p-6)/(2p)} \left(\|\nabla u\|_{L^{2}}+\|P\|_{L^{2}}\right)^{(6-p)/(2p)}. \label{uu2}
\end{gather}
\end{lemma}

\section{Global well-posedness for the pressureless system}\label{sec3}

In this section, we establish the global existence and uniqueness of strong
solutions to the Cauchy problem \eqref{PNS}, \eqref{a2}, and \eqref{a3}. The
argument is based on the local well-posedness result in Proposition
\ref{local}. Throughout this section, $(\rho,u)$ denotes a local strong solution
to this problem in $\mathbb{R}^{3}\times(0,T]$ for some $T>0$. We first treat
the case $\rho_{\infty}=0$.

\subsection{A \textit{ priori}  estimates I: scaling invariant estimates}\label{sub1}

In this subsection, we derive uniform-in-time lower-order estimates by a
bootstrap argument. These estimates play a crucial role in the proof of
Theorem \ref{th1}.

\begin{proposition}\label{pro}
	Assume that the hypotheses of Theorem \ref{th1} hold, and let
	$\varepsilon_{0}>0$ be the small constant appearing in \eqref{csca1}. Then there
	exist positive constants $c_{0}$ and $\varepsilon_{1}$, depending only on the
	physical parameters $\mu$, $\lambda$, and independent of the initial data and $T$, such that the
	following statement holds. If, for all $(x,t)\in\mathbb{R}^{3}\times[0,T]$,
	\begin{align}\label{cp1}
		\rho(x,t) \leq 2\bar{\rho},
		\qquad
		\bar{\rho}^{3}\|\sqrt{\rho}u(t)\|_{L^{2}}^{2}
		\|\nabla u(t)\|_{L^{2}}^{2}+\bar{\rho}^{3}\|\nabla u(t)\|_{L^{2}}^{2}\int_{0}^{T}\|\nabla u\|_{L^{2}}^{2}\,dt
		\leq 2\varepsilon
	\end{align}
	for some $\varepsilon$ satisfying
	\begin{align*}
		c_{0}\varepsilon_{0}\leq \varepsilon\leq \varepsilon_{1},
	\end{align*}
	then, for all $(x,t)\in\mathbb{R}^{3}\times[0,T]$,
	\begin{align*}
		\rho(x,t) \leq \frac{3}{2}\bar{\rho},
		\qquad
		\bar{\rho}^{3}\|\sqrt{\rho}u(t)\|_{L^{2}}^{2}
		\|\nabla u(t)\|_{L^{2}}^{2}+\bar{\rho}^{3}\|\nabla u(t)\|_{L^{2}}^{2}\int_{0}^{T}\|\nabla u\|_{L^{2}}^{2}\,dt
		\leq \frac{3}{2}\varepsilon.
	\end{align*}
	Here $\varepsilon_{1}$ and $c_{0}$  are chosen according to
	\eqref{xxxxx2}, \eqref{cs4}, and \eqref{c0}.
\end{proposition}

The proof of Proposition \ref{pro} is based on Lemmas \ref{rho} and \ref{sca}.
Throughout the rest of this subsection, $C$ denotes a generic positive constant
which is independent of the initial data and $T$, and may change from line to
line. We next collect several estimates that will be used frequently in the following
proofs.
\begin{lemma}\label{clemx}
For any $p\in[2,6]$, there exists a positive constant $C$ such that
	\begin{gather}
		\|\nabla u\|_{L^{6}}
		\leq C\| \rho \dot{u}\|_{L^{2}},\label{ceq4}\\[2mm]
		\|\nabla\operatorname{div}u\|_{L^{p}}
		+\|\nabla\operatorname{curl}u\|_{L^{p}}
		\leq C\| \rho \dot{u}\|_{L^{p}}.\label{ceq5}
	\end{gather}
\end{lemma}

\begin{proof}
	Applying $\operatorname{div}$ and $\operatorname{curl}$ to \eqref{PNS}$_2$, respectively, we derive
		\begin{align}\label{pns}
		\left\{
		\begin{aligned}
			&\mu\Delta(\operatorname{curl}u)=\operatorname{curl}(\rho\dot{u}),\\
			&(2\mu+\lambda)\Delta\operatorname{div}u=\operatorname{div}(\rho\dot{u}).
		\end{aligned}
		\right.
	\end{align}
	The standard $L^{p}$ elliptic estimate applied to the elliptic system
	\eqref{pns}  gives \eqref{ceq5}. In particular, taking $p=2$ in \eqref{ceq5}, we obtain
	\begin{align*}
		\|\nabla u\|_{L^{6}}
		\leq C\big(
		\|\operatorname{curl}u\|_{L^{6}}
		+\|\operatorname{div}u\|_{L^{6}}
		\big)
		\leq C\big(
		\|\nabla\operatorname{curl}u\|_{L^{2}}
		+\|\nabla\operatorname{div}u\|_{L^{2}}
		\big)
		\leq C\|\rho\dot{u}\|_{L^{2}},
	\end{align*}
	as the desired \eqref{ceq4}.
\end{proof}

We first derive the basic energy estimate.

\begin{lemma}\label{lem1}
	It holds that
	\begin{align}\label{X0}
		\sup_{0 \leq t \leq T}\|\sqrt{\rho} u(t)\|_{L^{2}}^{2}
		+\int_{0}^{T}\|\nabla u\|_{L^{2}}^{2}\,dt
		\leq C\|\sqrt{\rho_{0}}u_{0}\|_{L^{2}}^{2}.
	\end{align}
\end{lemma}

\begin{proof}
	Taking the $L^{2}$ inner product of \eqref{PNS}$_2$ with $u$, we obtain
	\begin{align*}
		\frac{1}{2}\frac{d}{dt}\|\sqrt{\rho}u\|_{L^{2}}^{2}
		+\mu\|\nabla u\|_{L^{2}}^{2}
		+(\mu+\lambda)\|\operatorname{div}u\|_{L^{2}}^{2}=0.
	\end{align*}
	Integrating this equality over $(0,T)$ gives \eqref{X0}.
\end{proof}

\begin{lemma}\label{lem2}
	Under the  assumption \eqref{cp1}, it holds that
	\begin{align}\label{xx0}
		\sup_{0 \leq t \leq T}\|\nabla u(t)\|_{L^{2}}^{2}
		+\int_{0}^{T}\|\sqrt{\rho}\dot{u}\|_{L^{2}}^{2}\,dt
		\leq
		C\|\nabla u_{0}\|_{L^{2}}^{2}
		+C\bar{\rho}^{3}\int_{0}^{T}\|\nabla u\|_{L^{2}}^{6}\,dt.
	\end{align}
\end{lemma}

\begin{proof}
	Taking the $L^{2}$ inner product of \eqref{PNS}$_2$ with $u_{t}$ and integrating	 by parts over $\mathbb{R}^{3}$, we get that
	\begin{align}\label{xxx2}
	\frac{1}{2}\frac{d}{dt}
		\Big( \mu\|\nabla u\|_{L^{2}}^{2}
		+(\mu+\lambda)\|\operatorname{div}u\|_{L^{2}}^{2}\Big)
		+\|\sqrt{\rho}\dot{u}\|_{L^{2}}^{2}
		&= \int \rho u\cdot\nabla u\cdot\dot{u} \, dx \notag\\
		&\leq \frac{1}{4}\|\sqrt{\rho}\dot{u}\|_{L^{2}}^{2}
		+C\int \rho |u|^{2}|\nabla u|^{2}\,dx.
	\end{align}
	For the last term in \eqref{xxx2}, it follows from H\"older's inequality, Sobolev's inequality, \eqref{cp1}, and \eqref{ceq5} that
	\begin{align}\label{xxx1}
		C\int \rho |u|^{2}|\nabla u|^{2}\,dx
		&\leq
		C\|\rho\|_{L^{\infty}}\|u\|_{L^{6}}^{2}
		\|\nabla u\|_{L^{3}}^{2} \notag\\
		&\leq
		C\bar{\rho}\|\nabla u\|_{L^{2}}^{2}
		\left( \|\operatorname{curl}u\|_{L^{3}}^{2}
		+\|\operatorname{div}u\|_{L^{3}}^{2}\right)  \notag\\[1mm]
		&\leq
		C\bar{\rho}\|\nabla u\|_{L^{2}}^{3}
		\big( \|\nabla\operatorname{curl}u\|_{L^{2}}
		+\|\nabla\operatorname{div}u\|_{L^{2}}\big) \notag\\
		&\leq
		C\bar{\rho}^{\frac{3}{2}}\|\nabla u\|_{L^{2}}^{3}
		\|\sqrt{\rho}\dot{u}\|_{L^{2}} \notag\\
		&\leq
		\frac{1}{4}\|\sqrt{\rho}\dot{u}\|_{L^{2}}^{2}
		+C\bar{\rho}^{3}\|\nabla u\|_{L^{2}}^{6}.
	\end{align}
	Substituting \eqref{xxx1} into \eqref{xxx2}, we obtain
	\begin{align*}
		\frac{d}{dt}
		\big( \mu\|\nabla u\|_{L^{2}}^{2}
		+(\mu+\lambda)\|\operatorname{div}u\|_{L^{2}}^{2}\big)
		+\|\sqrt{\rho}\dot{u}\|_{L^{2}}^{2}
		\leq C\bar{\rho}^{3}\|\nabla u\|_{L^{2}}^{6}.
	\end{align*}
	Integrating this inequality over $(0,T)$ yields \eqref{xx0}.
\end{proof}

Based on Lemmas \ref{lem1} and \ref{lem2}, we now close the
$L_{t}^{\infty}L_{x}^{2}$ estimate for $\nabla u$.

\begin{lemma}\label{co1}
	Under the  assumption  \eqref{cp1}, it holds that
	\begin{align}\label{xxxx1}
		\sup_{0 \leq t \leq T}\|\nabla u(t)\|_{L^{2}}^{2}
		+\int_{0}^{T}\|\sqrt{\rho}\dot{u}\|_{L^{2}}^{2}\,dt
		\leq C\|\nabla u_{0}\|_{L^{2}}^{2}.
	\end{align}
\end{lemma}

\begin{proof}
		Using \eqref{X0} and \eqref{xx0}, we have
	\begin{align}\label{k1}
		&\sup_{0 \leq t \leq T}\|\nabla u(t) \|_{L^{2}}^{2}+\int_{0}^{T}\|\sqrt{\rho}\dot{u} \|_{L^{2}}^{2}\, d t\notag\\
		&\leq  C \|\nabla u_{0} \|_{L^{2}}^{2} +C \bar{\rho}^{3}  \sup_{0 \leq t \leq T} \|\nabla u(t)\|_{L^{2}}^{2} \sup_{0 \leq t \leq T}\|\nabla u(t)\|_{L^{2}}^{2}\int_{0}^{T}\|\nabla u\|_{L^{2}}^{2}\,dt\notag\\
		&\leq  C \|\nabla u_{0} \|_{L^{2}}^{2} +\frac{1}{2}\sup_{0 \leq t \leq T}\|\nabla u(t)\|_{L^{2}}^{2},
	\end{align}
	where we have used the smallness assumption in \eqref{cp1} to ensure that
	\begin{align}\label{xxxxx2}
		C \bar{\rho}^{3} \sup_{0 \leq t \leq T} \|\nabla u(t)\|_{L^{2}}^{2}\int_{0}^{T}\|\nabla u\|_{L^{2}}^{2}\,dt  \leq 2 C \varepsilon \leq \frac{1}{2}.
	\end{align}	
Absorbing the last term on the right-hand side of \eqref{k1} into the left-hand side, we obtain \eqref{xxxx1}.
\end{proof}

We next derive a uniform upper bound for the density. This estimate is the key
ingredient for obtaining higher-order estimates and extending the local strong
solution globally.

\begin{lemma}\label{rho}
	Under the   assumption  \eqref{cp1}, it holds that, for any $(x,t)\in\mathbb{R}^{3}\times[0,T]$,
	\begin{align}\label{xz2}
		\rho(x,t)\leq \frac{3}{2}\bar{\rho}.
	\end{align}
\end{lemma}

\begin{proof}
	For any fixed $(x,t)\in\mathbb{R}^{3}\times[0,T]$, define
	\begin{align*}
		\rho^{\delta}(y,s)
		=
		\rho(y,s)
		+\delta\exp\left\{
		-\int_{0}^{s}\operatorname{div}u(X(\tau;x,t),\tau)\,d\tau
		\right\}>0,
	\end{align*}
	where $X(s;x,t)$ is the backward particle trajectory determined by
	\begin{align*}
		\left\{
		\begin{aligned}
			&\frac{d}{ds}X(s;x,t)=u(X(s;x,t),s),
			\qquad 0\leq s<t,\\
			&X(t;x,t)=x.
		\end{aligned}
		\right.
	\end{align*}
	By \eqref{PNS}$_1$, it is easy to verify that
	\begin{align*}
		\frac{d}{ds}\rho^{\delta}(X(s;x,t),s)
		+\rho^{\delta}(X(s;x,t),s)
		\operatorname{div}u(X(s;x,t),s)=0.
	\end{align*}
	Set
	\begin{align*}
		Y(s)=\ln\rho^{\delta}(X(s;x,t),s),
		\qquad
		b(s)=-\int_{0}^{s}\operatorname{div}u(X(\tau;x,t),\tau)\,d\tau.
	\end{align*}
	Then
	\begin{align}\label{co}
		Y'(s)=b'(s).
	\end{align}
	
	For  $\operatorname{div}u$, it follows from \eqref{PNS}$_2$ that
	\begin{align*}
		(2\mu+\lambda)\operatorname{div}u(X(\tau;x,t),\tau)
		=
		-\frac{d}{d\tau}
		\left[(-\Delta)^{-1}\operatorname{div}(\rho u)\right]
		+\left[u_{i},R_iR_j\right](\rho u_{j}),
	\end{align*}
	where
	\begin{align*}
	[u_{i},R_iR_j]=u_{i}R_iR_j-R_iR_j(u_{i}),
	\qquad
	R_iR_j=\partial_{i}(-\Delta)^{-1}\partial_{j}.
    \end{align*}
	It follows that
	\begin{align}\label{co11}
		b(t)-b(0)
		&\leq
		C\sup_{0\leq s\leq T}
		\left\|(-\Delta)^{-1}\operatorname{div}(\rho u)(s)\right\|_{L^{\infty}}
		+C\int_{0}^{t}
		\left\|[u_{i},R_{ij}](\rho u_{j})\right\|_{L^{\infty}}\,d\tau
		\notag\\
		&\leq
		C\bar{\rho}^{\frac{3}{4}}
		\|\sqrt{\rho_{0}}u_{0}\|_{L^{2}}^{\frac{1}{2}}
		\|\nabla u_{0}\|_{L^{2}}^{\frac{1}{2}}
		+C\bar{\rho}^{\frac{3}{2}}
		\|\sqrt{\rho_{0}}u_{0}\|_{L^{2}}
		\|\nabla u_{0}\|_{L^{2}}.
	\end{align}
	Indeed, one has that
	\begin{align}\label{inf}
		\left\|(-\Delta)^{-1}\operatorname{div}(\rho u)\right\|_{L^{\infty}}
		&\leq
		C\left\|(-\Delta)^{-1}\operatorname{div}(\rho u)\right\|_{L^{6}}^{\frac{1}{2}}
		\left\|\nabla(-\Delta)^{-1}\operatorname{div}(\rho u)\right\|_{L^{6}}^{\frac{1}{2}}\notag\\[1mm]
		&\leq
		C\|\rho u\|_{L^{2}}^{\frac{1}{2}}
		\|\rho u\|_{L^{6}}^{\frac{1}{2}} \\[1mm]
		&\leq
		C\bar{\rho}^{\frac{3}{4}}
		\|\sqrt{\rho_{0}}u_{0}\|_{L^{2}}^{\frac{1}{2}}
		\|\nabla u_{0}\|_{L^{2}}^{\frac{1}{2}}\notag,
	\end{align}
due to \eqref{X0} and \eqref{xxxx1}. Moreover, by the
	$L^{p}$-boundedness of the Riesz transforms, the commutator estimate in
	Lemma \ref{comm}, Gagliardo--Nirenberg inequality, and Sobolev's inequality,
	we have
	\begin{align}\label{cd}
		\left\|[u_{i},R_{ij}](\rho u_{j})\right\|_{L^{\infty}}
		&\leq
		C\|[u_{i},R_{ij}](\rho u_{j})\|_{L^{3}}^{\frac{1}{5}}
		\|\nabla [u_{i},R_{ij}](\rho u_{j})\|_{L^{4}}^{\frac{4}{5}}\notag\\[1mm]
		&\leq
		C\|u\|_{L^{6}}^{\frac{1}{5}}\|\rho u\|_{L^{6}}^{\frac{1}{5}}
		\|\nabla u\|_{L^{6}}^{\frac{4}{5}}\|\rho u\|_{L^{12}}^{\frac{4}{5}}\notag\\[1mm]
		&\leq
		C\bar{\rho}\|\nabla u\|_{L^{2}}\|\nabla u\|_{L^{6}}\notag\\[1mm]
		&\leq
		C\bar{\rho}\|\nabla u\|_{L^{2}}
		\big(\|\nabla\operatorname{curl}u\|_{L^{2}}
		+\|\nabla\operatorname{div}u\|_{L^{2}}\big)\notag\\
		&\leq
		C\bar{\rho}^{\frac{3}{2}}\|\nabla u\|_{L^{2}}
		\|\sqrt{\rho}\dot{u}\|_{L^{2}}.
	\end{align}
	Integrating \eqref{cd} in time and using \eqref{X0} and \eqref{xxxx1} give the second
	term in \eqref{co11}.
	
	Integrating \eqref{co} over $[0,t]$ and using \eqref{co11}, we obtain
	\begin{align}\label{cs4}
		\ln \rho^{\delta}(x,t)
		&\leq
		\ln(\bar{\rho}+\delta)+b(t)-b(0) \notag\\
		&\leq
		\ln(\bar{\rho}+\delta)
		+C\bar{\rho}^{\frac{3}{4}}
		\|\sqrt{\rho_{0}}u_{0}\|_{L^{2}}^{\frac{1}{2}}
		\|\nabla u_{0}\|_{L^{2}}^{\frac{1}{2}}
		+C\bar{\rho}^{\frac{3}{2}}
		\|\sqrt{\rho_{0}}u_{0}\|_{L^{2}}
		\|\nabla u_{0}\|_{L^{2}} \notag\\
		&\leq
		\ln(\bar{\rho}+\delta)+\ln\frac{3}{2},
	\end{align}
	where the smallness of $\varepsilon$ in \eqref{cp1} has been used in the last step. Letting $\delta\to0$ in \eqref{cs4} yields \eqref{xz2}.
\end{proof}

\begin{lemma}\label{sca}
	Under the assumption \eqref{cp1}, it holds that, for any $(x,t)\in\mathbb{R}^{3}\times[0,T]$,
	\begin{align*}
		\bar{\rho}^{3}	\|\sqrt{\rho}u(t)\|_{L^{2}}^{2}
		\|\nabla u(t)\|_{L^{2}}^{2}
		\leq \frac{3}{2}\varepsilon.
	\end{align*}
\end{lemma}

\begin{proof}
Combining \eqref{X0} multiplied by $\bar{\rho}^{3}$ with \eqref{xxxx1}, we have
	\begin{align*}
	 \bar{\rho}^{3}\|\sqrt{\rho}u(t)\|_{L^{2}}^{2}
	 \|\nabla u(t)\|_{L^{2}}^{2}+\bar{\rho}^{3}\|\nabla u(t)\|_{L^{2}}^{2}\int_{0}^{T}\|\nabla u\|_{L^{2}}^{2}\,dt \leq
		C\bar{\rho}^{3}
		\|\sqrt{\rho_{0}}u_{0}\|_{L^{2}}^{2}
		\|\nabla u_{0}\|_{L^{2}}^{2}
		\leq C\varepsilon_{0}.
	\end{align*}
	Choose
	\begin{align}\label{c0}
		c_{0}\triangleq \max\left\{\frac{1}{2},\frac{2C}{3}\right\}.
	\end{align}
	Since $c_{0}\varepsilon_{0}\leq \varepsilon$, we have
	$C\varepsilon_{0}\leq \frac{3}{2}\varepsilon$. This proves the desired estimate.
\end{proof}

\textbf{Proof of Proposition \ref{pro}.}
The density estimate in Lemma \ref{rho} and the scaling invariant estimate in
Lemma \ref{sca} close the bootstrap argument. Therefore, Proposition \ref{pro} follows.

\subsection{A \textit{ priori}  estimates II: time weighted estimates}\label{sub2}

Throughout this subsection, we assume that $(\rho,u)$ is a strong solution to the problem \eqref{PNS}, \eqref{a2}, and \eqref{a3} in $\mathbb{R}^{3}\times(0,T]$ for some $T>0$, and the bootstrap assumptions
\eqref{cp1} hold. Here, $C$ denotes a generic positive constant depending on the initial data and independent of $T$. We will derive higher-order estimates which are needed to establish the time continuity of the solution.

As a direct consequence of the estimates obtained in Subsection \ref{sub1}, we
have the following corollary.

\begin{corollary}\label{cor-lower}
	It holds that
	\begin{align}\label{czui1}
\sup_{0 \leq t \leq T}	\left( \|\sqrt{\rho}u(t)\|_{L^{2}}^{2}+\|\nabla u(t)\|_{L^{2}}^{2}
		+\|\rho(t)\|_{L^{\infty}}\right)  +\int_{0}^{T}\left( \|\nabla u\|_{L^{2}}^{2}+\|\sqrt{\rho}\dot{u}\|_{L^{2}}^{2}\right) \,dt	\leq C.
	\end{align}
\end{corollary}

Now, we establish the following time-weighted estimates for the solution.
\begin{lemma}\label{clem51}
Under the condition \eqref{czui1}, it holds that
	\begin{align}\label{cuu0}
		\sup _{0 \leq t \leq T}\left( t   \|\nabla u(t)\|_{L^{2}}^{2} \right)  +\int_{0}^{T} t  \|\sqrt{\rho}\dot{u}\|_{L^{2}}^{2}  \, d t \leq C.
	\end{align}
\end{lemma}

\begin{proof}
Using \eqref{xxx2}, \eqref{xxx1}, and \eqref{czui1}, one has that
	\begin{align}\label{cuu8}
	\frac{d}{d t}  \left(  \mu \|\nabla u\|_{L^{2}}^{2}+( \mu+\lambda)\|\operatorname{div} u\|_{L^{2}}^{2}\right) + \|\sqrt{\rho}\dot{u}\|_{L^{2}}^{2} \leq  C   \|\nabla u\|_{L^{2}}^{4}.
	\end{align}
	Then, multiplying \eqref{cuu8} by  $t$ and applying Gronwall's inequality, we obtain \eqref{cuu0}.
\end{proof}

\begin{lemma}\label{clem52}
	Under the condition \eqref{czui1}, it holds that
	\begin{align}\label{cll0}
		\sup _{0 \leq t \leq T} \left( t^{i} \| \sqrt{\rho} \dot{u}(t) \|_{L^{2}}^{2}\right) +\int_{0}^{T} t^{i} \| \nabla \dot{u} \|_{L^{2}}^{2} \, d t \leq C, \quad i=1,2.
	\end{align}
\end{lemma}
\begin{proof}
Applying the operator  $\partial_{t}+\operatorname{div}(u \cdot)$  to the  $j$-th component of  \eqref{PNS}$_{2}$  and multiplying the resultant equation by  $\dot{u}^{j}$, we obtain, after a direct calculation, that
	\begin{align}\label{cll5}
		&\frac{1}{2} \frac{d}{d t} \| \sqrt{\rho}\dot{u}\|_{L^{2}}^{2} \notag\\
		&=\mu \int \dot{u}^{j}\Big(\partial_{t} \Delta u^{j}+\operatorname{div}\big(u \Delta u^{j}\big)\Big) \, d x+(\mu+\lambda) \int \dot{u}^{j}\Big(\partial_{t} \partial_{j}(\operatorname{div} u)+\operatorname{div}\big(u \partial_{j}(\operatorname{div} u)\big)\Big) \, d x   \triangleq \sum_{i=1}^{2} K_{i}.
	\end{align}

	By a standard calculation, we first have
	\begin{align*}
		K_{1}
		&=-\mu \int\Big(\partial_{i} \dot{u}^{j} \partial_{t} \partial_{i} u^{j}
		+\Delta u^{j} u \cdot \nabla \dot{u}^{j}\Big) \, dx  \\
		&=-\mu \int\Big(
		|\nabla \dot{u}|^{2}
		-\partial_{i}\dot{u}^{j}u^{k}\partial_{k}\partial_{i}u^{j}
		-\partial_{i}\dot{u}^{j}\partial_{i}u^{k}\partial_{k}u^{j}
		+\Delta u^{j}u\cdot\nabla\dot{u}^{j}
		\Big)\,dx  \\
		&=-\mu \int\Big(
		|\nabla \dot{u}|^{2}
		+\partial_{i}\dot{u}^{j}\partial_{k}u^{k}\partial_{i}u^{j}
		-\partial_{i}\dot{u}^{j}\partial_{i}u^{k}\partial_{k}u^{j}
		-\partial_{i}u^{j}\partial_{i}u^{k}\partial_{k}\dot{u}^{j}
		\Big)\,dx.
	\end{align*}
	Hence, by H\"older's inequality and Young's inequality, one obtains
	\begin{align}\label{cll-k1}
		K_{1}
		\leq
		-\mu\|\nabla \dot{u}\|_{L^{2}}^{2}
		+C\int |\nabla \dot{u}|\,|\nabla u|^{2}\,dx
		\leq
		-\frac{7\mu}{8}\|\nabla \dot{u}\|_{L^{2}}^{2}
		+C\|\nabla u\|_{L^{4}}^{4}.
	\end{align}	
	For $K_{2}$, we use the identity
	\begin{align*}
		\operatorname{div}\dot{u}
	 =\partial_{j}\big(u_{t}^{j}+u^{k}\partial_{k}u^{j}\big)   =(\operatorname{div}u)_{t}
		+u^{k}\partial_{k}(\operatorname{div}u)
		+\partial_{j}u^{k}\partial_{k}u^{j}.
	\end{align*}
	Then
	\begin{align*}
		K_{2}
		&=(\mu+\lambda)\int \dot{u}^{j}
		\Big(
		\partial_{t}\partial_{j}(\operatorname{div}u)
		+\operatorname{div}\big(u\partial_{j}(\operatorname{div}u)\big)
		\Big)\,dx  \\
		&=-(\mu+\lambda)\int
		\Big(
		\partial_{j}\dot{u}^{j}(\operatorname{div}u)_{t}
		+u^{k}\partial_{k}\dot{u}^{j}\partial_{j}(\operatorname{div}u)
		\Big)\,dx  \\
		&=-(\mu+\lambda)\int |\operatorname{div}\dot{u}|^{2}\,dx
		+(\mu+\lambda)\int
		\operatorname{div}\dot{u}\,
		u^{k}\partial_{k}(\operatorname{div}u)\,dx  \\
		&\quad
		+(\mu+\lambda)\int
		\operatorname{div}\dot{u}\,
		\partial_{i}u^{k}\partial_{k}u^{i}\,dx
		-(\mu+\lambda)\int
		u^{k}\partial_{k}\dot{u}^{j}\partial_{j}(\operatorname{div}u)\,dx.
	\end{align*}
	For the two terms containing $u^{k}\partial_{k}(\operatorname{div}u)$ and
	$u^{k}\partial_{k}\dot{u}^{j}\partial_{j}(\operatorname{div}u)$, integration by parts gives
	\begin{align*}
		\int
		\operatorname{div}\dot{u}\,
		u^{k}\partial_{k}(\operatorname{div}u)\,dx
		-\int
		u^{k}\partial_{k}\dot{u}^{j}\partial_{j}(\operatorname{div}u)\,dx
		 =
		-\int (\operatorname{div}u)^{2}\operatorname{div}\dot{u}\,dx
		+\int
		\operatorname{div}u\,
		\partial_{j}u^{k}\partial_{k}\dot{u}^{j}\,dx.
	\end{align*}
	Therefore, one deduces from H\"older's inequality and Young's inequality that
	\begin{align}\label{cll-k2}
		K_{2}
		&=-(\mu+\lambda)\|\operatorname{div}\dot{u}\|_{L^{2}}^{2}
		+(\mu+\lambda)\int
		\operatorname{div}\dot{u}
		\Big(
		\partial_{i}u^{k}\partial_{k}u^{i}
		-(\operatorname{div}u)^{2}
		\Big)\,dx  \notag\\
		&\quad
		+(\mu+\lambda)\int
		\operatorname{div}u\,
		\partial_{j}u^{k}\partial_{k}\dot{u}^{j}\,dx\notag\\
		&\leq
		-(\mu+\lambda)\|\operatorname{div}\dot{u}\|_{L^{2}}^{2}
		+C\int |\operatorname{div}\dot{u}|\,|\nabla u|^{2}\,dx
		+C\int |\nabla\dot{u}|\,|\nabla u|^{2}\,dx   \notag\\
		&\leq
		-\frac{\mu+\lambda}{2}\|\operatorname{div}\dot{u}\|_{L^{2}}^{2}
		+\frac{\mu}{8}\|\nabla\dot{u}\|_{L^{2}}^{2}
		+C\|\nabla u\|_{L^{4}}^{4}.
	\end{align}	
	Substituting \eqref{cll-k1} and \eqref{cll-k2} into \eqref{cll5}, and using \eqref{czui1} and \eqref{ceq5}, we arrive at
	\begin{align}\label{cll6}
		\frac{d}{d t}\|\sqrt{\rho} \dot{u}\|_{L^{2}}^{2}
		+\|\nabla \dot{u}\|_{L^{2}}^{2}
		&\leq C\|\nabla u\|_{L^{4}}^{4} \notag\\
		&\leq C \|\nabla u \|_{L^{2}} \|\nabla u \|_{L^{6}}^{3} \notag\\[1mm]
		&\leq C \|\nabla u \|_{L^{2}}  \|\sqrt{\rho}\dot{u}\|_{L^{2}}^{3} \notag\\[1mm]
		&\leq C \|\sqrt{\rho}\dot{u}\|_{L^{2}}^{2}
		\left(
		\|\nabla u \|_{L^{2}}^{2}
		+\|\sqrt{\rho}\dot{u}\|_{L^{2}}^{2}
		\right).
	\end{align}

Multiplying \eqref{cll6} by $t$ and integrating the resultant over $(0,T)$, we obtain \eqref{cll0} with $i=1$ by Gronwall's inequality  and  \eqref{czui1}. Multiplying \eqref{cll6} by $t^{2}$,  we obtain
	\begin{align*}
		\frac{d}{d t} \left( t^{2}\|\sqrt{\rho} \dot{u}\|_{L^{2}}^{2}\right) +t^{2} \| \nabla \dot{u}  \|_{L^{2}}^{2}  &\leq Ct\|\sqrt{\rho} \dot{u}\|_{L^{2}}^{2}+  C t^{2}\|\sqrt{\rho}\dot{u}\|_{L^{2}}^{2}\left( \|\nabla u \|_{L^{2}}^{2} + \|\sqrt{\rho}\dot{u}\|_{L^{2}}^{2}\right).
	\end{align*}	
Using \eqref{czui1}, \eqref{cuu0}, and Gronwall's inequality, we obtain \eqref{cll0} with $i=2$.  This completes the proof.	
\end{proof}

\begin{lemma}\label{clem53}
	Under the condition \eqref{czui1}, it holds that
	\begin{gather}
		\sup _{0 \leq t \leq T} \left( t \| \nabla^{2} u(t)\|_{L^{2}}^{2}+\|  \rho_{t} (t) \|_{L^{2}}^{2}\right) +\int_{0}^{T} t\| \nabla u_{t}  \|_{L^{2}}^{2}\, d t \leq C,\label{clll1}	\\
		\sup _{0 \leq t \leq T}\left( \|\nabla \rho(t)\|_{L^{2}}+\|\nabla \rho(t)\|_{L^{q}}\right) +\int_{0}^{T}\left( \|\nabla^{2} u \|_{L^{q}}+ \|\nabla^{2} u \|_{L^{2}}+\|\nabla u\|_{L^{\infty}} \right)\, d t \leq C.\label{clll0}
	\end{gather}
\end{lemma}

\begin{proof}
Applying the spatial derivative $\nabla$ to the mass equation \eqref{PNS}$_1$, we obtain
	\begin{align*}
		0 =\partial_{t} \nabla \rho+u \cdot \nabla^{2} \rho+\nabla u \cdot \nabla \rho+\operatorname{div} u \nabla \rho+\rho \nabla \operatorname{div} u.
	\end{align*}
For $q\in(3,6)$, multiplying the above equality by $q|\nabla \rho|^{q-2}\nabla \rho$ and using \eqref{ceq5}, we get
	\begin{align*}
	 \frac{d}{d t}\|\nabla \rho\|_{L^{q}}^{q}& \leq C \int|\nabla u \| \nabla \rho|^{q} \, d x+C \int |\nabla \operatorname{div} u| |\nabla \rho|^{q-1} \, d x \\
		& \leq C\|\nabla u\|_{L^{\infty}}\|\nabla \rho\|_{L^{q}}^{q}+C\| \nabla \operatorname{div} u \|_{L^{q}}\|\nabla \rho \|_{L^{q}}^{q-1}\\[1mm]
			& \leq C\|\nabla u\|_{L^{\infty}}\|\nabla \rho\|_{L^{q}}^{q}+C\| \sqrt{\rho} \dot{u} \|_{L^{q}}\|\nabla \rho \|_{L^{q}}^{q-1},
	\end{align*}
which implies that
	\begin{align}\label{clll3}
		\frac{d}{d t}\|\nabla \rho\|_{L^{q}} \leq C\| \sqrt{\rho} \dot{u} \|_{L^{q}}+C\|\nabla u\|_{L^{\infty}}\|\nabla \rho\|_{L^{q}}.
	\end{align}
	By Gagliardo--Nirenberg inequality and \eqref{ceq5}, we have
	\begin{align*}
		\|\operatorname{div} u\|_{L^{\infty}} +\|\operatorname{curl} u\|_{L^{\infty}} & \leq C\left( \|\nabla \operatorname{div} u\|_{L^{2}}+\|\nabla \operatorname{curl} u\|_{L^{2}}+\|\nabla \operatorname{div} u\|_{L^{q}}+\| \nabla \operatorname{curl} u\|_{L^{q}}  \right)  \\
		& \leq C\left( \|\sqrt{\rho} \dot{u}\|_{L^{2}}+\| \sqrt{\rho} \dot{u} \|_{L^{q}}\right).
	\end{align*}
	Therefore, by \eqref{cinf0}, \eqref{ceq4}, and the elementary fact $\ln(e+\|\nabla \rho\|_{L^{q}})\geq 1$, we derive
	\begin{align}\label{clll4}
		\|\nabla u\|_{L^{\infty}} \leq C \ln\left( e+\|\nabla \rho\|_{L^{q}}\right) \left( \|\sqrt{\rho} \dot{u}\|_{L^{2}}+\| \sqrt{\rho} \dot{u} \|_{L^{q}}\right).
	\end{align}
	Substituting \eqref{clll4} into \eqref{clll3}, we find that, for any $q\in(3,6)$,
	\begin{align}\label{clll5}
		\frac{d}{d t} \ln \left( e+\|\nabla \rho\|_{L^{q}}\right)  \leq C \ln \left( e+\|\nabla \rho\|_{L^{q}}\right) \left( \|\sqrt{\rho} \dot{u}\|_{L^{2}}+\| \sqrt{\rho} \dot{u} \|_{L^{q}}\right).
	\end{align}

	To estimate the right-hand side of \eqref{clll5}, we set $\sigma(T)\triangleq \min\{1,T\}$. It infers from \eqref{czui1} and \eqref{cll0} that
	\begin{align}\label{clll6}
		\int_{0}^{T}\|\sqrt{\rho} \dot{u}\|_{L^{2}} \, d t&= \int_{0}^{\sigma(T)}\|\sqrt{\rho} \dot{u}\|_{L^{2}} \, d t + \int_{\sigma(T)}^{T} \|\sqrt{\rho} \dot{u}\|_{L^{2}} \, d t\notag\\
		&\leq C\left(\int_{0}^{\sigma(T)}\|\sqrt{\rho} \dot{u}\|_{L^{2}}^{2} \, d t\right)^{\frac{1}{2}}+C\left(\int_{\sigma(T)}^{T} t^{2}\|\nabla \dot{u}\|_{L^{2}}^{2}\, d t\right)^{\frac{1}{2}}\left(\int_{\sigma(T)}^{T} t^{-2}\, d t\right)^{\frac{1}{2}} \notag\\[2mm]
		&\leq C,
	\end{align}
and
\begin{align}\label{clll7}
	\int_{0}^{T}\|\sqrt{\rho}  \dot{u} \|_{L^{q}}\, d t&= \int_{0}^{\sigma(T)}\|\sqrt{\rho}  \dot{u} \|_{L^{q}}\, d t + \int_{\sigma(T)}^{T} \|\sqrt{\rho}  \dot{u} \|_{L^{q}} \, d t \notag\\
	&\leq C\int_{0}^{\sigma(T)}\| \sqrt{\rho}  \dot{u} \|_{L^{2}}^{\frac{6-q}{2q}}\| \sqrt{\rho}  \dot{u} \|_{L^{6}}^{\frac{3q-6}{2q}} \, d t + C\int_{\sigma(T)}^{T} \| \sqrt{\rho}  \dot{u} \|_{L^{2}}^{\frac{6-q}{2q}}\| \sqrt{\rho}  \dot{u} \|_{L^{6}}^{\frac{3q-6}{2q}} \, d t \notag\\
	&\leq C\left(\int_{0}^{\sigma(T)}\| \sqrt{\rho} \dot{u} \|_{L^{2}}^{2}\, d t\right)^{\frac{6-q}{4 q}}\left(\int_{0}^{\sigma(T)} t \| \nabla \dot{u} \|_{L^{2}}^{2} \, d t\right)^{\frac{3 q-6}{4 q}}\left(\int_{0}^{\sigma(T)} t^{-\frac{3 q-6}{2 q}}\, d t\right)^{\frac{1}{2}} \notag\\
	&\quad+C\left(\int_{\sigma(T)}^{T}t \| \sqrt{\rho} \dot{u} \|_{L^{2}}^{2} \, d t\right)^{\frac{6-q}{4 q}}\left(\int_{\sigma(T)}^{T} t^{2}\left\| \nabla \dot{u} \right\|_{L^{2}}^{2} \, d t\right)^{\frac{3 q-6}{4 q}}\left(\int_{\sigma(T)}^{T} t^{-\frac{5q-6}{2q}}\, d t\right)^{\frac{1}{2}}\notag\\[2mm]
	&\leq C.
\end{align}
Thus, one deduces from \eqref{clll5}, Gronwall's inequality, \eqref{clll6}, and \eqref{clll7} that
	\begin{align}\label{clll9}
		\sup _{0 \leq t \leq T}\|\nabla \rho\|_{L^{q}} \leq C.
	\end{align}
	This combined with \eqref{clll4}, \eqref{clll6},  and \eqref{clll7} yields 
	\begin{align}\label{clll10}
		\int_{0}^{T}\|\nabla u\|_{L^{\infty}}\, d t \leq C.
	\end{align}
Similarly, one can show that
	\begin{align}\label{clll11}
		\sup _{0 \leq t \leq T}\|\nabla \rho\|_{L^{2}} \leq C.
	\end{align}
	
	Next, from \eqref{PNS}$_1$, \eqref{do1}, \eqref{czui1}, \eqref{clll9}, \eqref{clll11}, and Gagliardo--Nirenberg inequality, we have
	\begin{align}
	 \|\rho_{t} \|_{L^{2}}  \leq C\big( \|\nabla \rho\|_{L^{3}}\|u\|_{L^{6}}+ \|\rho\|_{L^{\infty}}\|\nabla u\|_{L^{2}}\big)  \leq C\|\nabla u\|_{L^{2}},\label{do}
		\end{align}
		\begin{equation*}
		  \|  \nabla^{2} u \|_{L^{2}} +\|\nabla^{2} u \|_{L^{q}}  \leq C\left( \| \sqrt{\rho} \dot{u}\|_{L^{2}}+\| \sqrt{\rho} \dot{u} \|_{L^{q}}\right),
		\end{equation*}
		  \begin{align*}
		  \|\nabla u_{t} \|_{L^{2}}  &\leq C\left( \|\nabla \dot{u}\|_{L^{2}}+ \|u\|_{L^{\infty}} \|\nabla^{2} u \|_{L^{2}}+ \|\nabla u\|_{L^{4}}^{2}\right)  \\
		    &  \leq C\left( \|\nabla \dot{u}\|_{L^{2}}+ \|\nabla u\|_{L^{2}}^{\frac{1}{2}}\left\| \sqrt{\rho} \dot{u} \right\|_{L^{2}}^{\frac{3}{2}}\right),
	    \end{align*}
and
	\begin{align*}
	\int_{0}^{T}	t \|\nabla u_{t} \|_{L^{2}}^{2}\, d t  \leq C\int_{0}^{T}	t\|\nabla \dot{u}\|_{L^{2}}^{2}\, d t+C\sup _{0 \leq t \leq T}\left( t \| \sqrt{\rho} \dot{u}  \|_{L^{2}}^{2}\right)  \int_{0}^{T}  \| \sqrt{\rho} \dot{u}  \|_{L^{2}}\, d t.
	\end{align*}
	Combining the above estimates with \eqref{czui1}, \eqref{cll0}, \eqref{clll6}, and \eqref{clll7}, we obtain \eqref{clll1} and \eqref{clll0}.
\end{proof}

\subsection{Proof of Theorem \ref{th1}}\label{sub3}

Using the estimates obtained in Subsections \ref{sub1} and \ref{sub2}, we now prove Theorem \ref{th1} for the case $\rho_{\infty}=0$.

\textbf{Global well-posedness.} By Proposition \ref{local}, there exists
$T_{*}>0$ such that the Cauchy problem \eqref{PNS}, \eqref{a2}, and \eqref{a3}
admits a unique strong solution $(\rho,u)$ on
$\mathbb{R}^{3}\times(0,T_{*}]$. We shall extend this local strong solution
globally in time.

First, it follows from \eqref{xz}, \eqref{csca1}, and the definition of
$c_{0}$ in \eqref{c0} that the bootstrap assumptions \eqref{cp1} hold at the
initial time. Hence, by the time continuity of the local strong solution, there
exists $T_{1}\in(0,T_{*}]$ such that \eqref{cp1} holds on $[0,T_{1}]$.

Next, we define
\begin{align*}
	T^{*}\triangleq \sup \left\{ T>0 \,\middle|\,
	(\rho, u) \text{ is a strong solution of } \eqref{PNS},\eqref{a2},\eqref{a3}
	\text{ on } \mathbb{R}^{3}\times[0,T] \text{ satisfying } \eqref{cp1}
	\right\}.
\end{align*}
Then $T^{*}\geq T_{1}>0$. For any $0<\tau<T\leq T^{*}$ with $T<\infty$, it
follows from \eqref{czui1} and \eqref{clll1} that
\begin{align}\label{ccon1}
	u \in C\left([\tau,T];D^{1}\right),
\end{align}
 where we have used the standard embedding
\begin{align*}
	L^{\infty}\left(\tau, T ; H^{1}\right) \cap H^{1}\left(\tau, T ; H^{-1}\right) \hookrightarrow C\left([\tau, T] ; L^{s}\right), \quad \text { for any } s \in[2,6).
\end{align*}

Moreover, by \eqref{czui1} and \eqref{cuu0}, we have
\begin{align*}
	\rho u_{t}
	=\sqrt{\rho}\,\sqrt{\rho}u_{t}
	\in L^{2}\left(0,T;L^{2}\right).
\end{align*}
From \eqref{czui1}, \eqref{clll1},  \eqref{clll0}, and Sobolev's inequality, we
also obtain
\begin{align*}
	\rho_{t}u\in L^{2}\left(0,T;L^{2}\right).
\end{align*}
Therefore,
\begin{align}\label{pu1}
	(\rho u)_{t}
	=\rho u_{t}+\rho_{t}u
	\in L^{2}\left(0,T;L^{2}\right).
\end{align}
Furthermore, it follows from \eqref{czui1} that
\begin{align*}
	 \rho u =\sqrt{\rho}\,\sqrt{\rho}u \in L^{\infty}\left(0,T;L^{2}\right).
	 \end{align*}
	 This together with \eqref{pu1} yields
	 \begin{align}\label{ccon2}
	 	\rho u\in C\left([0,T];L^{2}\right).
	 	\end{align}

Next,  it follows from \eqref{clll10} and the classical result in
\cite[Chapter 3, Lemma 1.4]{Tem24}, we have
\begin{align*}
	\rho
	\in C\left([0,T];L^{2}\cap L^{q}\right)
	\cap C\left([0,T];H^{1}\cap W^{1,q}\text{-weak}\right).
\end{align*}
It remains to prove the strong continuity of $\rho$ in $H^{1}\cap W^{1,q}$.
From \eqref{PNS}$_1$, we know that, for $r=2$ or $r=q$, $|\nabla\rho|^r$ satisfies
\begin{align*}
\frac{d}{dt}\|\nabla\rho\|_{L^r}^r
\leq
C\|\nabla u\|_{L^\infty}\|\nabla\rho\|_{L^r}^r
+
C\|\nabla^2u\|_{L^r}
\|\nabla\rho\|_{L^r}^{r-1}.
\end{align*}
Therefore, we obtain
\begin{align*}
\frac{d}{dt}\|\nabla\rho \|_{L^r}
\leq
C\|\nabla u \|_{L^\infty}\|\nabla\rho \|_{L^r}
+
C\|\nabla^2u \|_{L^r}.
\end{align*}
By Gronwall's inequality, we have
\begin{align*}
	\|\nabla\rho(t)\|_{L^{r}}
	&\leq
	\left(
	\|\nabla\rho_{0}\|_{L^{r}}
	+C\int_{0}^{t}\|\nabla^{2}u(s)\|_{L^{r}}\,ds
	\right)
	\exp\left(
	C\int_{0}^{t}\|\nabla u(s)\|_{L^{\infty}}\,ds
	\right).
\end{align*}
In particular,
\begin{align*}
	\limsup_{t\rightarrow 0+}\|\nabla\rho(t)\|_{L^{r}}
	\leq \|\nabla\rho_{0}\|_{L^{r}}.
\end{align*}
Together with the weak continuity stated above, this implies that $\rho$ is
right-continuous in $W^{1,r}$ at $t=0$. Since \eqref{PNS}$_1$ is invariant under
time translations, the same argument applies at any time $t\in[0,T]$. Hence
\begin{align}\label{on2}
	\rho\in C\left([0,T];H^{1}\cap W^{1,q}\right).
\end{align}

Finally, suppose by contradiction that $T^{*}<\infty$. It follows from
\eqref{ccon1}, \eqref{ccon2},  and \eqref{on2} that the limit
\begin{align*}
	(\rho,u)(T^{*})
	=\lim_{t\rightarrow T^{*}}(\rho,u)(t)
\end{align*}
exists in the corresponding strong solution spaces and satisfies the local well-posedness initial
condition \eqref{sou} at $t=T^{*}$. Taking $(\rho,u)(T^{*})$ as new initial data,
Proposition \ref{local} allows us to extend the strong solution beyond $T^{*}$. This
contradicts the definition of $T^{*}$. Therefore, the strong
solution exists globally in time.

\textbf{Exponential stability.} Combining  Lemma \ref{lem1} with \eqref{cuu8}, we get
\begin{align}\label{cuu90}
	&\frac{d}{d t} \Big( \|\sqrt{\rho} u\|_{L^{2}}^{2}+\mu\|\nabla u \|_{L^{2}}^{2}+(\mu+\lambda)\|\operatorname{div} u \|_{L^{2}}^{2}\Big)  +\|\sqrt{\rho}\dot{u}\|_{L^{2}}^{2} + \|\nabla u \|_{L^{2}}^{2}\notag\\[2mm]
	&\leq   C   \|\nabla u\|_{L^{2}}^{2} \left( \|\nabla u\|_{L^{2}}^{2} + \|\sqrt{\rho} u\|_{L^{2}}^{2} \right).
\end{align}

Next, multiplying \eqref{PNS}$_1$ by
$\frac{3}{2}\rho^{\frac{1}{2}}$ and integrating the resultant over $\mathbb{R}^{3}$, we obtain
\begin{align*}
\frac{d}{dt}\|\rho \|_{L^{\frac{3}{2}}}^{\frac{3}{2}}
\leq C\|\operatorname{div}u \|_{L^\infty}
\|\rho \|_{L^{\frac{3}{2}}}^{\frac{3}{2}},
\end{align*}
which together with Gronwall's inequality leads to
\begin{align*}
\|\rho(t)\|_{L^{\frac{3}{2}}}
\leq C\|\rho_0\|_{L^{\frac{3}{2}}}
\exp\left(
C\int_0^t\|\operatorname{div}u(s)\|_{L^\infty}\,ds
\right).
\end{align*}
This along with \eqref{clll10} gives that
\begin{align}\label{ccon3}
	\rho\in L^{\infty}\left(0,T;L^{\frac{3}{2}}\right).
\end{align}

By \eqref{clll10}  and \eqref{ccon3}, we have
\begin{align}\label{mm1}
	\|\sqrt{\rho} u\|_{L^{2}}^{2}+\mu\|\nabla u \|_{L^{2}}^{2}+(\mu+\lambda)\|\operatorname{div} u \|_{L^{2}}^{2}\leq C_{1}    \|\nabla u\|_{L^{2}}^{2}.
\end{align}
Therefore, by \eqref{czui1}, one has that
\begin{align*}
	\int_{0}^{T}  \Big( \|\sqrt{\rho} u\|_{L^{2}}^{2}+\mu\|\nabla u \|_{L^{2}}^{2}+(\mu+\lambda)\|\operatorname{div} u \|_{L^{2}}^{2}\Big)   \, dt  \leq C.
\end{align*}
Choose $\alpha>0$ such that
\begin{align*}
	\alpha  \leq  \frac{1}{2 C_{1}}.
\end{align*}
Multiplying \eqref{cuu90} by  $e^{\alpha t}$ yields
\begin{align*}
	&\frac{d}{d t} \Big[ e^{\alpha  t}\Big(  \|\sqrt{\rho} u\|_{L^{2}}^{2}+\mu\|\nabla u \|_{L^{2}}^{2}+(\mu+\lambda)\|\operatorname{div} u \|_{L^{2}}^{2}\Big)\Big]   +e^{\alpha t}\left( \|\sqrt{\rho}\dot{u}\|_{L^{2}}^{2} + \|\nabla u \|_{L^{2}}^{2}\right)  \\[2mm]
	&\leq \alpha e^{\alpha   t}\Big(  \|\sqrt{\rho} u\|_{L^{2}}^{2}+\mu\|\nabla u \|_{L^{2}}^{2}+(\mu+\lambda)\|\operatorname{div} u \|_{L^{2}}^{2}\Big)+  C   e^{\alpha   t}\|\nabla u\|_{L^{2}}^{2} \left( \|\nabla u\|_{L^{2}}^{2} + \|\sqrt{\rho} u\|_{L^{2}}^{2} \right)\\[2mm]
	&\leq \frac{1}{2}e^{\alpha   t} \|\nabla u \|_{L^{2}}^{2} +  C   e^{\alpha   t}\|\nabla u\|_{L^{2}}^{2} \left( \|\nabla u\|_{L^{2}}^{2} + \|\sqrt{\rho} u\|_{L^{2}}^{2} \right).
\end{align*}	
Thus, by Gronwall's inequality and \eqref{czui1}, we obtain
\begin{align}\label{ce1}
	\sup _{0 \leq t \leq T}\Big[e^{\alpha t}\big(\|\sqrt{\rho} u(t)\|_{L^{2}}^{2} +\|\nabla u(t)\|_{L^{2}}^{2}\big)\Big] +\int_{0}^{T} e^{\alpha t}\big( \|\sqrt{\rho}\dot{u}\|_{L^{2}}^{2}+\|\nabla u\|_{L^{2}}^{2} \big)\, d t \leq C.
\end{align}

It follows from \eqref{cll0} that
\begin{align*}
	\limsup _{t \rightarrow 1} \|\sqrt{\rho} \dot{u}\|_{L^{2}}^{2} \leq C.
\end{align*}	
For  $t \geq 1$, multiplying \eqref{cll6} by  $e^{\alpha t}$, we get
\begin{align*}
	\frac{d}{d t}\left( e^{\alpha t}\|\sqrt{\rho} \dot{u}\|_{L^{2}}^{2}\right) +   e^{\alpha t}\|\nabla \dot{u}\|_{L^{2}}^{2}  &\leq \alpha e^{\alpha t}\|\sqrt{\rho} \dot{u}\|_{L^{2}}^{2}+Ce^{\alpha t}\|\sqrt{\rho}\dot{u}\|_{L^{2}}^{2}\left( \|\nabla u \|_{L^{2}}^{2} + \|\sqrt{\rho}\dot{u}\|_{L^{2}}^{2}\right).
\end{align*}
Then, it follows from Gronwall's inequality, \eqref{czui1}, and \eqref{ce1} that, for $t\geq 1$,
\begin{align}\label{ce2}
	\sup _{1 \leq t \leq T}\left( e^{\alpha t}\| \sqrt{\rho} \dot{u}(t)  \|_{L^{2}}^{2}\right)  +\int_{1}^{T} e^{\alpha t}\|  \nabla \dot{u}  \|_{L^{2}}^{2}\, d t \leq C.
\end{align}
Finally, the exponential stability stated in \eqref{cexp} follows directly
from \eqref{do1}, \eqref{do}, \eqref{ce1},  and \eqref{ce2}.

\textbf{Proof of Theorem \ref{th1} for $\rho_{\infty}>0$.} We now consider the case $\rho_{\infty}>0$. Since
\begin{align*}
\rho_{\infty}u_{0}=(\rho_{\infty}-\rho_{0})u_{0}+\rho_{0}u_{0},
\end{align*}
we have
\begin{align*}
	\rho_{\infty}^{2}\|u_{0}\|_{L^{2}}^{2}
	&\leq
	C\|(\rho_{\infty}-\rho_{0})u_{0}\|_{L^{2}}^{2}
	+C\|\rho_{0}u_{0}\|_{L^{2}}^{2} \\
	&\leq
	C\|\rho_{\infty}-\rho_{0}\|_{L^{3}}^{2}
	\|u_{0}\|_{L^{6}}^{2}
	+C\|\rho_{0}u_{0}\|_{L^{2}}^{2} \\
	&\leq
	C\|\rho_{\infty}-\rho_{0}\|_{L^{3}}^{2}
	\|\nabla u_{0}\|_{L^{2}}^{2}
	+C\|\rho_{0}u_{0}\|_{L^{2}}^{2}.
\end{align*}
Moreover, since
$\rho_{0}-\rho_{\infty}\in H^{1}\cap W^{1,q}$ with $q>3$, we have
$\rho_{0}\in L^{\infty}$ and $u_{0}\in L^{2}$. Hence
\begin{align*}
\|\sqrt{\rho_{0}}u_{0}\|_{L^{2}}^{2}
\leq
\|\rho_{0}\|_{L^{\infty}}\|u_{0}\|_{L^{2}}^{2}
\leq C.
\end{align*}
Meanwhile, the estimates obtained in Subsections \ref{sub1} and \ref{sub2}
remain valid in the present case.

Set
\begin{align*}
a\triangleq\rho-\rho_{\infty}.
\end{align*}
Then $a$ satisfies
\begin{align*}
a_t+u\cdot\nabla a+(a+\rho_{\infty})\operatorname{div}u=0.
\end{align*}
We now derive the general $L^p$-estimate of $a$ for $2\leq p\leq q$.
Since
\begin{align*}
a_t+u\cdot\nabla a+a\operatorname{div}u
=-\rho_\infty\operatorname{div}u.
\end{align*}
Multiplying the above equation by $|a|^{p-2}a$ and integrating over
$\mathbb R^3$, we obtain
\begin{align*}
	\frac{d}{dt}\|a\|_{L^p}
	\leq
	C \|\nabla u\|_{L^\infty}\|a\|_{L^p}
	+
	C  \|\nabla u\|_{L^p}.
\end{align*}
By Gronwall's inequality, one has that
\begin{align*}
	\sup_{0\leq t\leq T}\|a(t)\|_{L^{p}}
	&\leq
	\left(
	\|a_{0}\|_{L^{p}}
	+C T^{\frac{1}{2}}
	\left(\int_{0}^{T}\|\nabla u(s)\|_{L^{2}}^{2}\,ds\right)^{\frac{1}{2}}
	+C
	\int_{0}^{T}\|\nabla^{2}u(s)\|_{L^{2}}\,ds
	\right)\\
	&\quad\times
	\exp\left(
	 \int_{0}^{T}\|\nabla u(s)\|_{L^{\infty}}\,ds
	\right)
	\leq C(T).
\end{align*}

Then, by the standard transport equation argument and the estimates
\eqref{czui1}, \eqref{clll1}, and \eqref{clll0}, we have
\begin{align*}
a\in C\left([0,T];L^{2}\cap L^{q}\right)
\cap C\left([0,T];H^{1}\cap W^{1,q}\text{-weak}\right).
\end{align*}
The strong continuity of $a$ in $H^{1}\cap W^{1,q}$ can be proved in the same
way as in the case $\rho_{\infty}=0$. We therefore omit the details and conclude
that
\begin{align}\label{aaa}
	a=\rho-\rho_\infty
	\in C\left([0,T];H^{1}\cap W^{1,q}\right).
\end{align}

We now finish the continuation argument in the case $\rho_{\infty}>0$. Suppose,
by contradiction, that the maximal time $T^{*}$ is finite. From  \eqref{ccon1} and \eqref{aaa}, we have
$(\rho,u)(T^{*})$ satisfies the assumptions \eqref{sou} of the local well-posedness
result. Taking $(\rho,u)(T^{*})$ as new initial data, Proposition \ref{local} gives a
local strong solution beyond $T^{*}$. This contradicts the definition of $T^{*}$. Therefore, the strong solution exists globally in time in the case $\rho_{\infty}>0$.

\section{Global well-posedness for pressure system}\label{sec4}
In this section, we establish the global existence and uniqueness of strong
solutions to system \eqref{NS}. The argument is based on the local
well-posedness result stated in Proposition \ref{local}. In what follows, $(\rho,u )$ denotes a local strong solution to \eqref{NS}, \eqref{ca2}, and \eqref{ca3} in $\mathbb{R}^{3}\times(0,T]$ for some $T>0$.  First, we derive some a {\it priori} estimates for  $(\rho,u )$ and write $C(T)$  to emphasize that $C$ depends on $T$.
\begin{proposition}\label{wen}
Under the assumption \eqref{sca1}, there exists a positive constant $C$,
depending on the initial data but independent of $T$, such that
	\begin{gather}
		\sup_{0 \leq t \leq T}\Big(\| \sqrt{\rho} u(t)\|_{L^{2}}^{2}+\| \nabla u (t)\|_{L^{2}}^{2}\Big)+\int_{0}^{T}\left(\|\nabla u\|_{L^{2}}^{2} +\| \sqrt{\rho}\dot{u}\|_{L^{2}}^{2} \right) \, d t \leq C,\label{wen1}\\
		\sup_{0 \leq t \leq T}\Big(\|P(t) \|_{L^{1}} +\|P(t) \|_{L^{2}}^{2}+\|\rho(t) \|_{L^{\infty}}\Big)+\int_{0}^{T}\left( \|P \|_{L^{2}}^{2}+ \|P \|_{L^{3}}^{3}\right) \, d t \leq C.\label{wen2}
	\end{gather}
\end{proposition}

\begin{proof}
Here we only prove the estimate for $\int_{0}^{T}\|P\|_{L^{2}}^{2}\,dt$. The other
estimates follow from \cite[Proposition 2.2, Lemma 2.1, and Corollary 2.7]{W25}.

By \eqref{NS}$_1$ and the pressure law $P=\rho^{\gamma}$,  $P$
satisfies
\begin{align}\label{P}
P_{t}+\operatorname{div}(P u)+(\gamma-1)P\operatorname{div}u=0.
\end{align}
On the other hand, \eqref{NS}$_2$ gives
\begin{align*}
P=(-\Delta)^{-1} \operatorname{div}(\rho u)_{t}+(-\Delta)^{-1} \operatorname{div} \operatorname{div}(\rho u \otimes u)+(2 \mu+\lambda) \operatorname{div} u.
\end{align*}
Combining this identity with \eqref{P}, we obtain
\begin{align}\label{PP}
	\|P\|_{L^{2}}^{2}= & \frac{d}{d t} \int(-\Delta)^{-1} \operatorname{div}(\rho u) P \, d x+(\gamma-1) \int(-\Delta)^{-1} \operatorname{div}(\rho u) P \operatorname{div} u \, d x \notag\\
	& -\int P u \cdot \nabla(-\Delta)^{-1} \operatorname{div}(\rho u) \, d x+\int(-\Delta)^{-1} \operatorname{div} \operatorname{div}(\rho u \otimes u) Pd x +(2 \mu+\lambda) \int P \operatorname{div} u d x \notag\\
	\leq & \frac{d}{d t} \int(-\Delta)^{-1} \operatorname{div}(\rho u) P \, d x+C\left\|(-\Delta)^{-1} \operatorname{div}(\rho u)\right\|_{L^{\infty}}\|P\|_{L^{2}}\|\nabla u\|_{L^{2}} \notag\\
	& +C\|P\|_{L^{\frac{3}{2}}}\|u\|_{L^{6}}^{2}+C\|P\|_{L^{2}}\|\nabla u\|_{L^{2}} \notag\\
	\leq & \frac{d}{d t} \int(-\Delta)^{-1} \operatorname{div}(\rho u) P \, d x+\frac{1}{2}\|P\|_{L^{2}}^{2}+C\|\nabla u\|_{L^{2}}^{2},
\end{align}
where we have used  \eqref{inf}, \eqref{wen1}, \eqref{wen2}, and
\begin{align*}
	\sup_{0 \leq t \leq T} \|(-\Delta)^{-1} \operatorname{div}(\rho u) (t)\|_{L^{\infty}} \leq C \sup_{0 \leq t \leq T}\Big( \|\sqrt{\rho} u (t)\|_{L^{2}}+\|\nabla u(t)\|_{L^{2}}\Big) \leq C.
\end{align*}
Integrating \eqref{PP} over $(0,T)$ yields
\begin{align*}
	\int_{0}^{T}\|P\|_{L^{2}}^{2}\, d t \leq  C\sup_{0 \leq t \leq T}\big(  \|(-\Delta)^{-1} \operatorname{div}(\rho u)(t) \|_{L^{\infty}}\|P(t)\|_{L^{1}}\big) +C\int_{0}^{T}\|\nabla u\|_{L^{2}}^{2}\, d t\leq C.
\end{align*}
This completes the proof.
\end{proof}

\subsection{A  \textit{ priori}  estimates}
Throughout this subsection, we derive several higher-order estimates which are
needed to establish the time continuity of the solution. We first prove the
following time-weighted estimate.
\begin{lemma}\label{lem51}
	Under the condition \eqref{sca1}, it holds that
	\begin{align}\label{uu0}
		\sup _{0 \leq t \leq T}\Big( t  \|\nabla u(t)\|_{L^{2}}^{2} +t\|P(t)\|_{L^{2}}^{2}\Big)  +\int_{0}^{T} t \left(  \|\sqrt{\rho} \dot{u} \|_{L^{2}}^{2}+\|P\|_{L^{3}}^{3}\right)   \, d t \leq C.
	\end{align}
\end{lemma}

\begin{proof}
 Multiplying \eqref{NS}$_2$ by $u_t$, using \eqref{P}, and integrating by parts
 over $\mathbb{R}^{3}$, we obtain
	\begin{align}\label{c2}
		& \frac{1}{2} \frac{d}{d t}  \big(\mu\|\nabla u\|_{L^{2}}^{2}+(\mu+\lambda)\|\operatorname{div} u\|_{L^{2}}^{2}\big)+\|\sqrt{\rho}\dot{u}\|_{L^{2}}^{2} \notag\\[2mm]
		&=  \frac{d}{d t} \int P \operatorname{div} u \, dx-\int P_{t} \operatorname{div} u \, dx +\int \rho u \cdot \nabla u \cdot \dot{u} \, dx\notag\\
		&=  \frac{d}{d t} \int P \operatorname{div} u \, dx-\frac{1}{2(2 \mu+\lambda)} \frac{d}{d t} \int P^{2} \, dx-\frac{1}{2 \mu+\lambda} \int P u \cdot \nabla F \, dx\notag\\
		& \quad +\frac{\gamma-1}{2 \mu+\lambda} \int P \operatorname{div} u F \, dx+ \int \rho u \cdot \nabla u \cdot \dot{u} \, dx\notag\\
		&\triangleq \frac{d}{d t} \left( \int P \operatorname{div} u \, dx-\frac{1}{2(2 \mu+\lambda)}  \|P \|_{L^{2}}^{2}\right)  +\sum_{i=1}^{3} I_{i}.
	\end{align}
By H\"older's inequality, \eqref{wen2}, and Lemma \ref{F}, we have
	\begin{align*}
	I_{1}+I_{2} & \leq   C\|u\|_{L^{6}} \|P \|_{L^{3}}\|\nabla F\|_{L^{2}}+C \|P \|_{L^{3}}\|\operatorname{div} u\|_{L^{2}}\|F\|_{L^{6}} \\[2mm]
	& \leq C\|\nabla u\|_{L^{2}} \|P \|_{L^{3}}  \|\sqrt{\rho} \dot{u}\|_{L^{2}}\\
	& \leq \frac{1}{3}\|\sqrt{\rho}\dot{u}\|_{L^{2}}^{2}+C\|\nabla u\|_{L^{2}}^{2} \|P \|_{L^{3}}^{2},\\[2mm]
		I_{3} & \leq \frac{1}{12}\|\sqrt{\rho}\dot{u}\|_{L^{2}}^{2}+ C \int \rho|u|^{2}|\nabla u|^{2}\, dx\\
		& \leq \frac{1}{12}\|\sqrt{\rho}\dot{u}\|_{L^{2}}^{2}+C \|u\|_{L^{6}}^{2}\|\nabla u\|_{L^{3}}^{2} \\
		& \leq \frac{1}{12}\|\sqrt{\rho}\dot{u}\|_{L^{2}}^{2}+C  \|\nabla u\|_{L^{2}}^{2}\left(\|\operatorname{curl} u\|_{L^{3}}^{2}+\|F\|_{L^{3}}^{2}+ \|P \|_{L^{3}}^{2}\right) \\
		& \leq \frac{1}{12}\|\sqrt{\rho}\dot{u}\|_{L^{2}}^{2}+C  \|\nabla u\|_{L^{2}}^{2}\big(\|\operatorname{curl} u\|_{L^{2}}\|\nabla \operatorname{curl} u\|_{L^{2}}+\|F\|_{L^{2}}\|\nabla F\|_{L^{2}}+\left\|P\right\|_{L^{3}}^{2}\big)\\
		&\leq \frac{1}{12}\|\sqrt{\rho}\dot{u}\|_{L^{2}}^{2}+C  \|\nabla u\|_{L^{2}}^{2}\left(\|\nabla u\|_{L^{2}}+ \|P\|_{L^{2}}\right)\|\sqrt{\rho} \dot{u}\|_{L^{2}}+C  \|\nabla u\|_{L^{2}}^{2} \|P \|_{L^{3}}^{2}\\
		&\leq \frac{1}{6}\|\sqrt{\rho}\dot{u}\|_{L^{2}}^{2}+C  \|\nabla u\|_{L^{2}}^{4}\left(\|\nabla u\|_{L^{2}}^{2}+ \|P\|_{L^{2}}^{2}\right) +C  \|\nabla u\|_{L^{2}}^{2} \|P \|_{L^{3}}^{2}.
	\end{align*}
Substituting these estimates $I_{i}$ into \eqref{c2}, we get
	\begin{align}\label{oq}
		 \frac{d}{d t}A(t)   +\| \sqrt{\rho}\dot{u}\|_{L^{2}}^{2}
		&\leq C  \|\nabla u\|_{L^{2}}^{4}\big(\|\nabla u\|_{L^{2}}^{2}+\|P\|_{L^{2}}^{2}\big)+C  \|\nabla u\|_{L^{2}}^{2}\|P\|_{L^{3}}^{2}\notag\\
		&\leq C  \|\nabla u\|_{L^{2}}^{4}\big(\|\nabla u\|_{L^{2}}^{2}+\|P\|_{L^{2}}^{2}\big)+ C_{2}\|P \|_{L^{3}}^{3},
	\end{align}
	where
	\begin{align*}
	A(t) \triangleq   \mu\|\nabla u\|_{L^{2}}^{2}+(\mu+\lambda)\|\operatorname{div} u\|_{L^{2}}^{2}-2\int   P \operatorname{div} u \, dx+\frac{1}{ 2 \mu+\lambda }  \|P \|_{L^{2}}^{2}.
    \end{align*}

    It follows from \eqref{P} that
    \begin{align}\label{x1}
    	P_t+u\cdot \nabla P+\gamma P \operatorname{div} u=0.
    \end{align}
    For $p \geq 2$, multiplying \eqref{x1} by $p P^{p-1}$ and integrating the resulting equality over $\mathbb{R}^{3}$ gives
    \begin{align*}
    \frac{d}{d t}\|P\|_{L^{p}}^{p} +\frac{p \gamma-1}{2 \mu+\lambda}\|P\|_{L^{p+1}}^{p+1}=-\frac{p \gamma-1}{2 \mu+\lambda} \int P^{p} F \, d x.
    \end{align*}
    By H\"older's inequality, this implies that
    \begin{align}\label{x2}
    	\frac{d}{d t}\|P\|_{L^{p}}^{p}+\frac{p \gamma-1}{2(2 \mu+\lambda)}\|P\|_{L^{p+1}}^{p+1} \leq C \|F\|_{L^{p+1}}^{p+1}.
    \end{align}
	
	For $A(t)$, we have
	\begin{align*}
	\|\nabla u\|_{L^{2}}^{2}-C\|P\|_{L^{2}}^{2} \leq A(t)  \leq   C\left( \|\nabla u\|_{L^{2}}^{2}+ \|P\|_{L^{2}}^{2}\right).
	\end{align*}
	Choose
	\begin{align*}
	C_{3}\geq \frac{2(2\mu+\lambda)(C_{2}+1)}{2\gamma-1}
\end{align*}
	large enough so that
	\begin{align}\label{oo1}
	 \|\nabla u\|_{L^{2}}^{2}+\|P\|_{L^{2}}^{2} \leq A(t)+C_{3}\|P\|_{L^{2}}^{2} \leq   C\left( \|\nabla u\|_{L^{2}}^{2}+ \|P\|_{L^{2}}^{2}\right).
\end{align}
Taking $p=2$ in \eqref{x2}, adding $C_{3}$ times the resulting inequality to
\eqref{oq}, and using \eqref{uu2} and Proposition \ref{wen}, we derive
	\begin{align*}
		&\frac{d}{d t}  \left(A(t)+C_{3}\|P\|_{L^{2}}^{2}\right)+\| \sqrt{\rho}\dot{u}\|_{L^{2}}^{2} +\| P\|_{L^{3}}^{3}\\
		&\leq C  \|\nabla u\|_{L^{2}}^{2}\left(\|\nabla u\|_{L^{2}}^{2}+ \|P \|_{L^{2}}^{2}\right) + C\|F \|_{L^{3}}^{3}\\
		&\leq   \frac{1}{2}\| \sqrt{\rho}\dot{u}\|_{L^{2}}^{2}+\|\nabla u\|_{L^{2}}^{2}\left(\|\nabla u\|_{L^{2}}^{2}+ \|P \|_{L^{2}}^{2}\right)+C  \left( \|\nabla u\|_{L^{2}}+\|P \|_{L^{2}}\right) ^{6}\\
		&\leq \frac{1}{2}\| \sqrt{\rho}\dot{u}\|_{L^{2}}^{2}+C  \left(\|\nabla u\|_{L^{2}}^{2}+ \|P \|_{L^{2}}^{2}\right)\left(\|\nabla u\|_{L^{2}}^{2}+ \|P \|_{L^{2}}^{2}\right).
	\end{align*}
		Since the first term on the right-hand side can be absorbed by the left, we have
		\begin{align}\label{o1}
		 \frac{d}{d t}  \left(A(t)+C_{3}\|P\|_{L^{2}}^{2}\right)+\| \sqrt{\rho}\dot{u}\|_{L^{2}}^{2} +\| P\|_{L^{3}}^{3} \leq C  \left(\|\nabla u\|_{L^{2}}^{2}+ \|P \|_{L^{2}}^{2}\right)\left(\|\nabla u\|_{L^{2}}^{2}+ \|P \|_{L^{2}}^{2}\right).
		\end{align}
Multiplying \eqref{o1} by  $t$ and noting that
\begin{align*}
	\int_{0}^{T}  \left( A(t)+C_{3}\|P\|_{L^{2}}^{2}\right)   \, dt \leq C,
\end{align*}
we conclude from Proposition \ref{wen}, \eqref{o1}, and
Gronwall's inequality that \eqref{uu0} holds.
\end{proof}

\begin{lemma}\label{lem52}
	Under condition \eqref{sca1}, it holds that
	\begin{align}\label{ll0}
		\sup _{0 \leq t \leq T}\left(  t^{i} \| \sqrt{\rho}\dot{u}(t)\|_{L^{2}}^{2}+t^{i}\| P(t)\|_{L^{3}}^{3}\right) +\int_{0}^{T} t^{i}\left(\|\nabla \dot{u}\|_{L^{2}}^{2}+\|P\|_{L^{4}}^{4}\right) \, d t \leq C, \quad i=1,2.
	\end{align}
\end{lemma}
\begin{proof}
 Applying the operator $\partial_{t}+\operatorname{div}(u\,\cdot)$ to the
 $j$-th component of \eqref{NS}$_2$ and multiplying the resulting equation by
 $\dot{u}^{j}$, we obtain that
	\begin{align}\label{ll5}
	\frac{1}{2} \frac{d}{d t} \| \sqrt{\rho}\dot{u}\|_{L^{2}}^{2}
		&=\mu \int \dot{u}^{j}\Big(\partial_{t} \Delta u^{j}+\operatorname{div}\big(u \Delta u^{j}\big)\Big) \, d x+(\mu+\lambda) \int \dot{u}^{j}\Big(\partial_{t} \partial_{j}(\operatorname{div} u)+\operatorname{div}\big(u \partial_{j}(\operatorname{div} u)\big)\Big) \, d x \notag\\
		&\quad-\int \dot{u}^{j}\Big(\partial_{j} P_{t}+\operatorname{div} (u \partial_{j}P )\Big) \, d x \triangleq \sum_{i=1}^{3} K_{i}.
	\end{align}
The estimates for $K_{1}$ and $K_{2}$ are the same as those in Lemma
\ref{clem52}. It remains to estimate $K_{3}$. Using
\begin{gather*}
	P_t+\div(Pu)=-(\gamma-1)P\,\div u,\\[1mm]
	\partial_j P_t+\partial_j\div(Pu)
	=-(\gamma-1)\partial_j(P\,\div u),\\[1mm]
	\partial_j P_t+\div(\partial_j P\,u)
	=-(\gamma-1)\partial_j(P\,\div u)-\partial_k(P\partial_j u^k),
\end{gather*}
we have
	\begin{align*}
		K_{3} & =(\gamma-1)\int \dot u^j\,\partial_j(P\,\div u)\,dx
		+\int \dot u^j\,\partial_k(P\partial_j u^k)\,dx\\
		& =-(\gamma-1)\int P\,\div u\,\div\dot u\,dx
		-\int P\,\partial_j u^k\,\partial_k\dot u^j\,dx \\
		& \leq \frac{\mu}{8}\|\nabla \dot{u}\|_{L^{2}}^{2}+C \|\nabla u \|_{L^{4}}^{4}+C \|P\|_{L^{4}}^{4}.
	\end{align*}

 Substituting the estimates for $K_i$ into \eqref{ll5} yields that
	\begin{align}\label{ll6}
		\frac{d}{d t}\|\sqrt{\rho} \dot{u}\|_{L^{2}}^{2}+ \|\nabla \dot{u}\|_{L^{2}}^{2}  \leq C\left( \|P \|_{L^{4}}^{4}+\|\nabla u \|_{L^{4}}^{4}\right).
 	\end{align}	
 Multiplying \eqref{ll6} by $t$ and using \eqref{uu3}, we get
 	\begin{align}\label{cll7}
 	\frac{d}{d t} \left( t\|\sqrt{\rho} \dot{u}\|_{L^{2}}^{2}\right) +  t \|\nabla \dot{u}\|_{L^{2}}^{2}\leq \|\sqrt{\rho} \dot{u}\|_{L^{2}}^{2}+C t \left( \|F\|_{L^{4}}^{4}+  \|\operatorname{curl} u\|_{L^{4}}^{4}\right) +C_{4} t \|P\|_{L^{4}}^{4}.
 	\end{align}
 	
 Set
 \begin{align*}
 \kappa\triangleq \frac{2(2\mu+\lambda)(C_{4}+1)}{3\gamma-1}.
\end{align*}
 Taking $p=3$ in \eqref{x2} and adding $\kappa t$ times the resulting inequality
 to \eqref{cll7}, we obtain
 	\begin{align*}
 		&\frac{d}{d t}\left(t \|\sqrt{\rho} \dot{u}\|_{L^{2}}^{2}+\kappa t \|P\|_{L^{3}}^{3}\right) + t\|\nabla \dot{u}\|_{L^{2}}^{2}+t\|P\|_{L^{4}}^{4} \\
 		& \leq C \left(\|\sqrt{\rho} \dot{u}\|_{L^{2}}^{2}+\|P\|_{L^{3}}^{3} \right)+C t\left( \|F\|_{L^{4}}^{4}+ \|\operatorname{curl} u\|_{L^{4}}^{4}\right)  \\[1mm]
 		& \leq  C\left(\|\sqrt{\rho} \dot{u}\|_{L^{2}}^{2}+\|P\|_{L^{3}}^{3} \right)+C t\|\sqrt{\rho} \dot{u}\|_{L^{2}}^{3}\left(\|\nabla u\|_{L^{2}}+\|P\|_{L^{2}}\right) \\[1mm]
 		& \leq  C \left(\|\sqrt{\rho} \dot{u}\|_{L^{2}}^{2}+\|P\|_{L^{3}}^{3} \right)+C t\|\sqrt{\rho} \dot{u}\|_{L^{2}}^{2}\left(\|\sqrt{\rho} \dot{u}\|_{L^{2}}^{2}+\|\nabla u\|_{L^{2}}^{2}+\|P\|_{L^{2}}^{2}\right)\\[1mm]
 		& \leq  C \left(\|\sqrt{\rho} \dot{u}\|_{L^{2}}^{2}+\|P\|_{L^{3}}^{3} \right)+C t\left(\|\sqrt{\rho} \dot{u}\|_{L^{2}}^{2}+\|P\|_{L^{3}}^{3} \right)\left(\|\sqrt{\rho} \dot{u}\|_{L^{2}}^{2}+\|\nabla u\|_{L^{2}}^{2}+\|P\|_{L^{2}}^{2}\right),
 	\end{align*}
 where \eqref{uu2} has been used in the second inequality. Combining this
 estimate with Gronwall's inequality and Proposition \ref{wen}, we obtain
 \eqref{ll0} with $i=1$.

Similarly, multiplying \eqref{ll6} by $t^{2}$ and adding $\kappa t^{2}$ times
\eqref{x2} with $p=3$, we get
\begin{align*}
		&\frac{d}{d t}\left(t^{2} \|\sqrt{\rho} \dot{u}\|_{L^{2}}^{2}+\kappa t^{2}\|P\|_{L^{3}}^{3}\right) +  t^{2}\|\nabla \dot{u}\|_{L^{2}}^{2}+t^{2}\|P\|_{L^{4}}^{4} \\[1mm]
		& \leq C t \left(\|\sqrt{\rho} \dot{u}\|_{L^{2}}^{2}+\|P\|_{L^{3}}^{3} \right)+C t^{2}\left( \|F\|_{L^{4}}^{4}+ \|\operatorname{curl} u\|_{L^{4}}^{4}\right)  \\[1mm]
		& \leq  C t \left(\|\sqrt{\rho} \dot{u}\|_{L^{2}}^{2}+\|P\|_{L^{3}}^{3} \right)+C t^{2}\left(\|\sqrt{\rho} \dot{u}\|_{L^{2}}^{2}+\|P\|_{L^{3}}^{3} \right)\left(\|\sqrt{\rho} \dot{u}\|_{L^{2}}^{2}+\|\nabla u\|_{L^{2}}^{2}+\|P\|_{L^{2}}^{2}\right).
\end{align*}
Using  Proposition \ref{wen}, \eqref{uu0}, and
Gronwall's inequality, we obtain \eqref{ll0} with $i=2$.
\end{proof}

\begin{lemma}\label{lem53}
Under condition \eqref{sca1}, there exists a positive
constant $C(T)$ such that
	\begin{gather*}
	\sup _{0 \leq t \leq T} \left( t \| \nabla^{2} u(t)\|_{L^{2}}^{2}+\|  \rho_{t} (t) \|_{L^{2}}^{2}\right) +\int_{0}^{T} t\| \nabla u_{t}  \|_{L^{2}}^{2} d t \leq C(T),\\
	\sup _{0 \leq t \leq T}\left( \|\nabla \rho(t)\|_{L^{2}}+\|\nabla \rho(t)\|_{L^{q}}\right) +\int_{0}^{T}\left( \|\nabla^{2} u \|_{L^{q}}+ \|\nabla^{2} u \|_{L^{2}}+\|\nabla u\|_{L^{\infty}}\right) d t \leq C(T).
\end{gather*}
\end{lemma}

\begin{proof}
Applying the spatial derivative $\nabla$ to the mass equation \eqref{NS}$_1$,
we obtain
	\begin{align*}
\partial_{t} \nabla \rho+u \cdot \nabla^{2} \rho+\nabla u \cdot \nabla \rho+\operatorname{div} u \nabla \rho+\rho \nabla \operatorname{div} u=0.
	\end{align*}
	For $q \in(3,6)$, multiplying the above equality by $q|\nabla \rho|^{q-2} \nabla \rho$  gives that
	\begin{align}\label{lll3}
		\frac{d}{d t}\|\nabla \rho\|_{L^{q}} \leq C\|\nabla^{2} u \|_{L^{q}}+C \|\nabla u\|_{L^{\infty}} \|\nabla \rho\|_{L^{q}}
		\leq C\|\sqrt{\rho}\dot{u} \|_{L^{q}}+C\left( \|\nabla u\|_{L^{\infty}}+1\right)\|\nabla \rho\|_{L^{q}},
	\end{align}
where we have used the elliptic estimate
\begin{align}\label{ell}
	\|\nabla^{2}u\|_{L^{q}}
	\leq
	C\left(
	\|\sqrt{\rho}\dot{u}\|_{L^{q}}
	+\|\nabla P\|_{L^{q}}
	\right),
\end{align}
which follows from the standard $L^{p}$ theory for the elliptic system
\begin{align*}
	\mu\Delta u+(\mu+\lambda)\nabla\operatorname{div}u
	=\rho\dot{u}+\nabla P,
	\qquad
	u\to0\quad\text{as } |x|\to\infty.
\end{align*}

Using \eqref{uu1} and Gagliardo--Nirenberg inequality, we have
	\begin{align*}
		\|\operatorname{div} u\|_{L^{\infty}} +\|\operatorname{curl} u\|_{L^{\infty}} & \leq C \big(\| F\|_{L^{\infty}}+\|  P\|_{L^{\infty}}+\| \operatorname{curl} u\|_{L^{\infty}}\big) \\
		& \leq C\Big( \|\nabla F\|_{L^{q}}^{\frac{3q}{5q-6}}+\| \nabla \operatorname{curl} u\|_{L^{q}}^{\frac{3q}{5q-6}}+\|P\|_{L^{\infty}} \Big) \\
		& \leq C\left( \|\sqrt{\rho} \dot{u}\|_{L^{q}}^{\frac{3q}{5q-6}}+1\right).
	\end{align*}	
This combined with \eqref{cii} implies that
	\begin{align}\label{lll4}
			\|\nabla u\|_{L^{\infty}} & \leq C\left(\|\operatorname{div} u\|_{L^{\infty}}+\|\operatorname{curl} u\|_{L^{\infty}}\right) \ln \left(e+ \|\nabla^{2} u \|_{L^{q}}\right)+C\|\nabla u\|_{L^{2}}+C \notag\\
			& \leq C\left(1+\|\sqrt{\rho} \dot{u}\|_{L^{q}}^{\frac{3q}{5q-6}}\right) \ln \left(e+\|\nabla \rho\|_{L^{q}}+\|\sqrt{\rho} \dot{u}\|_{L^{q}}\right)+C\notag\\
			& \leq C\left(1+\|\sqrt{\rho}\dot{u}\|_{L^{q}}\right) \ln \left(e+\|\nabla \rho\|_{L^{q}}\right).
	\end{align}	
	Substituting \eqref{lll4} into \eqref{lll3}, we find that, for any
	$q\in(3,6)$,
	\begin{align}\label{lll5}
		\frac{d}{d t} \ln \left( e+\|\nabla \rho\|_{L^{q}}\right)  \leq C \ln \big(e+\|\nabla \rho\|_{L^{q}}\big)\left( 1+\| \sqrt{\rho}\dot{u} \|_{L^{q}}\right).
	\end{align}
 It deduces from Proposition \ref{wen}, Lemmas \ref{lem51}, and \ref{lem52} that
 \begin{align}\label{lll7}
 	\int_{0}^{T}\|\sqrt{\rho}  \dot{u} \|_{L^{q}}\, d t
 	&\leq C\left(\int_{0}^{\sigma(T)}\| \sqrt{\rho} \dot{u} \|_{L^{2}}^{2}\, d t\right)^{\frac{6-q}{4 q}}\left(\int_{0}^{\sigma(T)} t \| \nabla \dot{u} \|_{L^{2}}^{2} \, d t\right)^{\frac{3 q-6}{4 q}}\left(\int_{0}^{\sigma(T)} t^{-\frac{3 q-6}{2 q}}\, d t\right)^{\frac{1}{2}} \notag \\
 	&\quad+C\left(\int_{\sigma(T)}^{T}t \| \sqrt{\rho} \dot{u} \|_{L^{2}}^{2} \, d t\right)^{\frac{6-q}{4 q}}\left(\int_{\sigma(T)}^{T} t^{2} \| \nabla \dot{u}  \|_{L^{2}}^{2} \, d t\right)^{\frac{3 q-6}{4 q}}\left(\int_{\sigma(T)}^{T} t^{-\frac{5q-6}{2q}}\, d t\right)^{\frac{1}{2}}\notag \\
 	&\leq C.
 \end{align}
Thus, we derive from \eqref{lll5}, Gronwall's inequality, and \eqref{lll7} that
	\begin{align}\label{v1}
		\sup _{0 \leq t \leq T}\|\nabla \rho\|_{L^{q}} \leq C(T),
	\end{align}
	which together with \eqref{lll4} and \eqref{lll7} gives
	\begin{align}\label{v2}
		\int_{0}^{T}\|\nabla u\|_{L^{\infty}} d t \leq C(T).
	\end{align}
Similarly, one can obtain that
	\begin{align}\label{v3}
		\sup _{0 \leq t \leq T}\|\nabla \rho\|_{L^{2}} \leq C(T).
	\end{align}

	Next, from \eqref{NS}$_1$, \eqref{ell}, and Gagliardo--Nirenberg inequality, we have
	\begin{gather*}
		\|\rho_{t} \|_{L^{2}}  \leq C\big( \|\nabla \rho\|_{L^{3}}\|u\|_{L^{6}}+ \|\rho\|_{L^{\infty}}\|\nabla u\|_{L^{2}}\big)  \leq C\|\nabla u\|_{L^{2}}, \\[2mm]
		\|  \nabla^{2} u \|_{L^{2}} +\|\nabla^{2} u \|_{L^{q}}  \leq C\left( \| \sqrt{\rho} \dot{u}\|_{L^{2}}+\| \sqrt{\rho} \dot{u} \|_{L^{q}}+\|\nabla \rho\|_{L^{2}}+\|\nabla \rho\|_{L^{q}}\right),
	\end{gather*}
	 and
	\begin{align*}
		\|\nabla u_{t} \|_{L^{2}}  &\leq C\left( \|\nabla \dot{u}\|_{L^{2}}+ \|u\|_{L^{\infty}} \|\nabla^{2} u \|_{L^{2}}+ \|\nabla u\|_{L^{4}}^{2}\right)  \\
		&  \leq C\left( \|\nabla \dot{u}\|_{L^{2}}+ \|\nabla u\|_{L^{2}}^{\frac{1}{2}} \| \nabla^{2} u \|_{L^{2}}^{\frac{3}{2}}\right)\\
		&  \leq C\left( \|\nabla \dot{u}\|_{L^{2}}+ \|\nabla u\|_{L^{2}}^{\frac{1}{2}} \left( \| \sqrt{\rho} \dot{u}  \|_{L^{2}}+\|\nabla \rho\|_{L^{2}}\right) ^{\frac{3}{2}}\right).
	\end{align*}
	Combining the above estimates with \eqref{v1}--\eqref{v3}, we  complete the proof.
\end{proof}

\subsection{Proof of Theorem \ref{th2}}
The continuation argument from local to global solutions has already been
presented in detail in the proof of Theorem \ref{th1}. Therefore, we omit the
repetition and only point out that the estimates obtained above provide the
required bounds for extending the local strong solution. This completes the
proof of Theorem \ref{th2}.

\section*{Appendix}\label{app}

In this appendix, we derive the critical scaling transformation for the following compressible Navier--Stokes system in $\mathbb{R}^{3}$:
\begin{align*}
	\left\{
	\begin{aligned}
		&\rho_{t}+\operatorname{div}(\rho u)=0, \\[2mm]
		&(\rho u)_{t}+\operatorname{div}(\rho u\otimes u)
		-\mu\Delta u-(\lambda+\mu)\nabla\operatorname{div}u+\nabla P(\rho)=0.
	\end{aligned}
	\right.
\end{align*}
The far-field condition is imposed as
\begin{align*}
	\lim_{|x|\to\infty}(\rho,u)(t,x)=(0,0).
\end{align*}

We first describe the scaling with respect to the \textit{spatial normalized scale}. For $\ell>0$, let
\begin{align*}
	\rho^\ell(x,t)=\ell^a\rho(\ell x,\ell^c t),
	\qquad
	u^\ell(x,t)=\ell^b u(\ell x,\ell^c t).
\end{align*}
Substituting the above transformation into the continuity equation gives
\begin{align*}
c=b+1.
\end{align*}
For the momentum equation, the four main terms scale respectively as
\begin{align*}
\partial_t(\rho^\ell u^\ell)\sim \ell^{a+b+c}&,
\qquad
\operatorname{div}(\rho^\ell u^\ell\otimes u^\ell)\sim \ell^{a+2b+1},\\
\nabla P(\rho^\ell)\sim \ell^{a\gamma+1}&,
\qquad
\Delta u^\ell,\ \nabla\operatorname{div}u^\ell\sim \ell^{b+2}.
\end{align*}
Therefore, the invariance of the system requires
\begin{align*}
a+b+c=a+2b+1=a\gamma+1=b+2.
\end{align*}
Solving the above algebraic system yields
\begin{align*}
	a=\frac{2}{\gamma+1},\qquad
	b=\frac{\gamma-1}{\gamma+1},\qquad
	c=\frac{2\gamma}{\gamma+1}.
\end{align*}
Hence the spatially normalized scaling is
\begin{align}\label{space-scaling}
	\boxed{
		\rho^\ell(x,t)
		=
		\ell^{\frac{2}{\gamma+1}}
		\rho\left(\ell x,\ell^{\frac{2\gamma}{\gamma+1}}t\right),
		\qquad
		u^\ell(x,t)
		=
		\ell^{\frac{\gamma-1}{\gamma+1}}
		u\left(\ell x,\ell^{\frac{2\gamma}{\gamma+1}}t\right).
	}
\end{align}

Similarly, one may normalize the scaling by the time variable. For $\beta >0$, set
\begin{align*}
	\rho^\beta (x,t)=\beta ^a\rho(\beta ^c x,\beta  t),
	\qquad
	u^\beta (x,t)=\beta ^b u(\beta ^c x,\beta  t).
\end{align*}
In this case,   the \textit{time normalized scale} takes the form
\begin{align}\label{time-scaling}
	\boxed{
		\rho^\beta (x,t)
		=
		\beta ^{\frac{1}{\gamma}}
		\rho\left(\beta ^{\frac{\gamma+1}{2\gamma}}x,\beta  t\right),
		\qquad
		u^\beta (x,t)
		=
		\beta ^{\frac{\gamma-1}{2\gamma}}
		u\left(\beta ^{\frac{\gamma+1}{2\gamma}}x,\beta  t\right).
	}
\end{align}

The two formulations \eqref{space-scaling} and \eqref{time-scaling} describe the same scaling group. Indeed, if
\begin{align*}
\ell=\beta ^{\frac{\gamma+1}{2\gamma}},
\end{align*}
and the \textit{spatial normalized scale} \eqref{space-scaling} is transformed exactly into the \textit{time normalized scale} \eqref{time-scaling}.

\begin{remark}
	It is worth emphasizing that the amplitude factor in front of the density in
	the scaling transformation plays an essential role in constructing
	scaling invariant quantities.  These exponents determine how the density is rescaled and therefore must be
	taken into account when one constructs an initial quantity invariant under the
	flow scaling.
	
	However, if a quantity $Q(\rho,u)$ satisfies
	\begin{align*}
	Q(\rho^\ell,u^\ell)=Q(\rho,u),
	\qquad \text{for all } \ell>0,
    \end{align*}
	then it is also invariant under the time-normalized scaling. Indeed, for any
	$\beta >0$, we choose
	\begin{align*}
	\ell=\beta ^{\frac{\gamma+1}{2\gamma}}.
	 \end{align*}
	With this choice, the time normalized scaling coincides with the
	spatially normalized scaling. Hence
	\begin{align*}
	Q(\rho^\beta,u^\beta )
	=
	Q(\rho^\ell,u^\ell)
	=
	Q(\rho,u).
    \end{align*}
	Thus, invariance under the spatially normalized scaling is equivalent to
	invariance under the time normalized scaling.	
\end{remark}

\section*{Conflict of interest}
The authors have no conflicts to disclose.

\section*{Data availability}
No data was used for the research described in the article.



\end{document}